\documentclass[11pt,letterpaper]{amsart} 

\newtheorem{theorem}{Theorem}[section]
\newtheorem{lemma}[theorem]{Lemma}
\newtheorem{proposition}[theorem]{Proposition}
\newtheorem{corollary}[theorem]{Corollary}

\theoremstyle{definition}

\newtheorem{remark}[theorem]{Remark}

\numberwithin{equation}{section}

\newcommand{\R}{\mathbf{R}}
\newcommand{\eps}{\varepsilon}
\newcommand{\cE}{\mathcal{E}}
\DeclareMathOperator{\supp}{\operatorname{supp}}
\DeclareMathOperator{\dist}{\operatorname{dist}}

\begin{document}

\title[Blow-up rate in non-convex domains]{
Blow-up rate for the subcritical semilinear\\ heat equation
in non-convex domains}

\author[H. Miura]{Hideyuki Miura}
\address{Department of Mathematics, 
Institute of Science Tokyo, 
2-12-1 Ookayama, Meguro-ku, Tokyo 152-8551, Japan}
\email{hideyuki@math.titech.ac.jp}

\author[J. Takahashi]{Jin Takahashi}
\address{Department of Mathematical and Computing Science, 
Institute of Science Tokyo, 
2-12-1 Ookayama, Meguro-ku, Tokyo 152-8552, Japan}
\email{takahashi@comp.isct.ac.jp}

\author[E. Zhanpeisov]{Erbol Zhanpeisov}
\address{Mathematical Institute, Graduate School of Science, 
Tohoku University, 
6-3, Aramaki Aza-Aoba, Aoba-ku, Sendai 980-8578, Japan}
\email{zhanpeisov.erbol.d6@tohoku.ac.jp}


\begin{abstract}
We consider the semilinear heat equation 
$u_t=\Delta u+|u|^{p-1} u$ in possibly non-convex and unbounded domains. 
Our main result shows the nonexistence of 
type II blow-up for possibly sign-changing solutions 
in the energy subcritical range $(n-2)p<n+2$. 
This resolves a long-standing 
open question dating back to the 1980s 
and also deduces the blow-up of the scaling critical norm.
\end{abstract}

\maketitle

\section{Introduction}
We consider the following semilinear heat equation: 
\begin{equation}\label{eq:main}
\left\{ 
\begin{aligned}
	&u_t=\Delta u+|u|^{p-1} u  &&\mbox{ in }\Omega\times (0,T), \\
	&u=0 &&\mbox{ on }\partial\Omega\times(0,T), \\
	&u(\cdot,0)=u_0 &&\mbox{ in }\Omega, \\
\end{aligned}
\right.
\end{equation}
where $p>1$, $u_0\in L^\infty(\Omega)$, 
$\Omega$ is a domain in $\R^n$ with $n\geq1$ 
and the boundary condition is not imposed if $\Omega=\R^n$. 
This is a prototypical nonlinear parabolic problem 
for the study of blow-up phenomena, 
and a fundamental objective 
is to specify the blow-up rate at the blow-up time $t=T$.

In the 1980s, Weissler \cite{We84,We85} first obtained the blow-up rate 
\begin{equation}\label{eq:upper}
	\sup_{0<t<T} (T-t)^\frac{1}{p-1} \|u(\cdot,t)\|_{L^\infty(\Omega)} 
	<\infty
\end{equation}
for particular classes of nonnegative solutions 
in the subcritical range 
\[
	1<p<p_S:=
	\left\{ 
	\begin{aligned}
	&\frac{n+2}{n-2} &&\mbox{ for }n\geq3, \\
	&\infty &&\mbox{ for }n=1,2.    \\
	\end{aligned}
	\right.
\]
Here, \eqref{eq:upper} is referred to as type I blow-up 
and is consistent with the blow-up rate of the ODE $v_t=|v|^{p-1} v$. 
If \eqref{eq:upper} fails, the blow-up is said to be of type II. 
Friedman--McLeod \cite{FM85} also obtained the type I rate 
for nonnegative solutions in a bounded convex domain 
under the assumption $u_t\geq0$. 
By the groundbreaking work of Giga--Kohn \cite{GK87}, 
the type I rate holds in possibly unbounded convex domains 
for either $p<p_S$ and $u\geq0$, or $p<(3n+8)/(3n-4) (<p_S)$.

After the above pioneering works, 
the blow-up rate of solutions to \eqref{eq:main} has been extensively studied, 
see Quittner--Souplet \cite[Section 23]{QSbook2} and the references given there. 
Up to now, the type I rate for $p<p_S$ 
has been proved for all nonnegative solutions in general domains 
(see Quittner \cite{Qu21}), and 
for possibly sign-changing solutions in convex domains 
(see Giga--Matsui--Sasayama \cite{GMS04,GMS04s}). 
We note that the range $p<p_S$ cannot be extended, since 
there exist type II blow-up solutions \cite{dPMW19,dPMWZZpre,Ha20s,Sc12} 
for $p=p_S$ with $3\leq n\leq 6$. 
These constructions provide rigorous proofs 
of the prediction by Filippas--Herrero--Vel\'azquez \cite{FHV00}.
Thus, the remaining problem for $p<p_S$ is to establish the type I blow-up rate 
for solutions that may change sign 
in possibly unbounded and non-convex domains. 
This long-standing problem was listed 
in \cite[OP2.5, Appendix I]{QSbook2},  
and we settle it in the present paper.

\begin{theorem}\label{th:main}
Let $n\geq1$, $1<p<p_S$, 
$\Omega$ be any uniformly 
$C^{2+\alpha}$ domain in $\R^n$ with  $0<\alpha<1$ 
and $u$ be a classical solution of \eqref{eq:main} with $u_0\in L^\infty(\Omega)$. 
If the maximal existence time $T>0$ is finite, then 
\[
	\sup_{0<t<T} (T-t)^\frac{1}{p-1} \|u(\cdot,t)\|_{L^\infty(\Omega)} 
	<\infty. 
\]
\end{theorem}

Obtaining the type I rate 
is known as an important foothold for deriving refined properties 
of blow-up profiles \cite{FK92,GK89,HV93,KRZ11,MZ97,NZ17,Ve93c} and blow-up sets \cite{FI14,GK89,Za02}. 
As a more direct application, we focus on the behavior 
of the critical norm $L^{q_c}$, where $q_c:= n(p-1)/2$ 
and the $L^{q_c}$ norm is invariant under the scaling 
$u(x,t)\mapsto \lambda^{2/(p-1)} u(\lambda x, \lambda^2 t)$ for $\lambda>0$. 
A remarkable development by Mizoguchi--Souplet \cite{MS19} 
asserts that for all $p>1$, 
the type I blow-up implies the critical norm blow-up. 
Since the type II blow-up solutions obtained 
in \cite{dPMW19,dPMWZZpre,Sc12} 
for $p=p_S$ with $3\leq n\leq 5$
have bounded $L^{q_c}$ norm, 
the problem \cite[OP2.1, Appendix I]{QSbook2} asks about
the existence of blow-up solutions 
with bounded $L^{q_c}$ norm for $p\neq p_S$. 
Recently, the first and second authors \cite{MTapp} proved 
the nonexistence of such blow-up solutions 
for $p> p_S$ in general domains, 
see also \cite{BMT24pre,MT23pre}. 
However, the case $p<p_S$ in possibly non-convex domains is still open. 
Theorem \ref{th:main} together with \cite{MS19} 
deduces the following, which completely solves the problem.

\begin{corollary}\label{cor:CNB}
Let $n\geq1$ and $1<p<p_S$. 
Under the same assumptions as in Theorem \ref{th:main}, 
if the maximal existence time $T>0$ is finite, then 
\[
	\lim_{t\to T} \|u(\cdot,t)\|_{L^{q_c}(\Omega)} =\infty. 
\]
\end{corollary}

We note that Corollary \ref{cor:CNB} holds
even for $0<q_c<1$, equivalently, for $1<p<(n+2)/n (<p_S)$. 
In addition, $\|u(\cdot,t)\|_{L^{q_c}(\Omega)}$ is allowed to 
take the value $\infty$ in the statement. 
If we further assume $u_0\in L^{q_c}(\Omega)$ with $q_c>1$ for instance, 
then $\|u(\cdot,t)\|_{L^{q_c}(\Omega)}<\infty$ holds for $0\leq t<T$. 
For $p=p_S$ with $n=6$, 
a sign-changing type II blow-up solution 
with critical norm blow-up was constructed in \cite{Ha20s}.  
For $p=p_S$ with $n\geq7$, the type I rate is valid for all nonnegative solutions 
in convex domains \cite{WW21pre}, 
and so the critical norm blow-up also occurs in this setting.
We expect that our quasi-monotonicity formula 
(Proposition \ref{pro:qmono}, valid for all $p>1$) will be 
useful for removing the convexity assumption.

The proof of Theorem \ref{th:main} is based 
on the analysis of the Giga--Kohn weighted energy \cite{GK87} 
of solutions in the backward similarity variables. 
There are two main difficulties:
(i) the domain may be non-convex, 
where the usual monotonicity of the Giga--Kohn energy can fail; 
(ii) the solution may change sign, 
and so the Liouville type theorems 
for nonnegative solutions of the parabolic problem \cite{Qu21} or 
the elliptic problem \cite{GS81} are not available.

Our first step is to derive a quasi-monotonicity 
formula for the Giga--Kohn energy in Proposition \ref{pro:qmono}. 
Following Blatt--Struwe \cite{BS15b}, we introduce 
a boundary integral quantity in \eqref{eq:betardef}, 
and then we estimate it. 
This gives a uniform bound for the energy 
without assuming the convexity of the domain 
or any sign condition on the solution. 
While they \cite{BS15b} assume the boundedness of the domain 
and localize an analog of the Giga--Kohn energy 
by using a cut-off function 
that truncates both near and far from the spatial center of scaling, 
we address possibly unbounded domains 
and do not localize the energy at this stage. 

The next step is based on a bootstrap scheme 
of Giga--Matsui--Sasayama \cite{GMS04,GMS04s} type, 
which dates back to the bootstrap argument of Quittner \cite{Qu99}.
We apply a Cazenave--Lions \cite{CL84} type interpolation 
and the maximal regularity, 
with a bootstrap argument. 
We note that they \cite{GMS04,GMS04s} use $p<p_S$ in the bootstrap steps, 
whereas we do not. 
Consequently, 
we obtain an $L^\infty_t L^q_x$ estimate 
for all $p>1$ and $q<p+1$ in Proposition \ref{pro:lessp1}. 
It seems to be of independent interest. 
Finally, the interior and boundary regularity 
for linear parabolic equations 
with the help of $p<p_S$, equivalently of $q_c<p+1$, 
yields the desired type I blow-up rate.

This paper is organized as follows. 
In Section \ref{sec:mono}, we give the quasi-monotonicity formula 
for the Giga--Kohn energy (Proposition \ref{pro:qmono}). 
In Section \ref{sec:esti}, 
as applications of the monotonicity, 
we derive several estimates including the base step of 
the inductive estimates (Proposition \ref{pro:base}). 
In Section \ref{sec:rate}, we prove an estimate corresponding to 
each of the inductive steps (Proposition \ref{pro:boot}). 
Then by using the resultant inequality (Proposition \ref{pro:lessp1}), 
we show Theorem \ref{th:main}.

\section{Quasi-monotonicity of Giga--Kohn energy}\label{sec:mono}
Let $p>1$ and let $u$ 
be a solution of \eqref{eq:main} with blow-up time $0<T<\infty$. 
Since $\Omega$ is a uniformly  $C^{2+\alpha}$ domain, 
the parabolic regularity (see 
\cite[Theorems 5.1.11, 5.1.13 (iv)]{Lubook}
for instance) guarantees that
\begin{equation}\label{eq:regu}
\left\{ \begin{aligned}
	&\mbox{$u\in L^\infty((0,T_2); L^\infty(\Omega))$, 
	$\nabla u$, $\nabla^2 u$ and $u_t$ are continuous} \\
	&\mbox{and bounded on $\overline{\Omega}\times[T_1,T_2]$ for each $0<T_1<T_2<T$. }
\end{aligned} \right.
\end{equation}
We note that the uniformity of the regularity of the domain 
will also be used to derive 
the geometric inequality \eqref{eq:geom} below.

Let $\tilde x\in\overline{\Omega}$ and $0<\tilde t\leq T$. 
We define the Giga--Kohn weighted energy by 
\begin{equation}\label{eq:GKenedef}
\begin{aligned}
	E_{(\tilde x,\tilde t)}(t) 
	&:= (\tilde t-t)^{\frac{p+1}{p-1}} 
	\int_\Omega \left( \frac{|\nabla u|^2}{2} 
	- \frac{|u|^{p+1}}{p+1} 
	+ \frac{|u|^2}{2(p-1)(\tilde t-t)}
	\right) \\
	&\quad \times K_{(\tilde x,\tilde t)}(x,t) dx
\end{aligned}
\end{equation}
for $0< t<\tilde t$, 
where $K_{(\tilde x,\tilde t)}(x,t) = 
 (\tilde t -t)^{-n/2} e^{-|x-\tilde x|^2/(4(\tilde t-t))}$ 
is the backward heat kernel 
and $u$ and $\nabla u$ in the integral are evaluated at $(x,t)$. 
The goal of this section is to prove the following 
quasi-monotonicity formula for the Giga--Kohn energy of solutions 
in possibly non-convex domains.

\begin{proposition}\label{pro:qmono}
Let $p>1$. 
Then there exists $\tilde C>0$ depending only on $n$, $p$ and $\Omega$ such that
for all $\tilde x\in \overline{\Omega}$ and $T-\delta^2 \leq t' < t < T$, 
\begin{equation}\label{eq:qmonoes}
\begin{aligned}
	E_{(\tilde x,T)}(t) 
	&\leq 
	E_{(\tilde x,T)}(t')  
	+ \tilde C \delta^{\frac{2(p+1)}{p-1}+1}  
	\|\nabla u\|_{L^\infty(\Omega\times (\frac{T}{2}-\delta^2, T-\delta^2))}^2 \\
	&\quad + \tilde C \delta^{\frac{2(p+1)}{p-1}-1} 
	\|u\|_{L^\infty(\Omega\times (\frac{T}{2}-\delta^2, T-\delta^2))}^2, 
\end{aligned}
\end{equation}
where $0<\delta<\sqrt{T/2}$ depends only on $n$, $p$, $\Omega$ and $T$. 
\end{proposition}

We remark that the second and third terms 
in the right-hand side of \eqref{eq:qmonoes} 
are finite, since $u$ satisfies \eqref{eq:regu} and 
$t=T-\delta^2$ is prior to the blow-up time.

As preliminaries, 
we define the backward rescaled solution and the backward similarity variables by 
\begin{equation}\label{eq:backw}
\begin{aligned} 
	&w_{(\tilde x,\tilde t)}(\eta,\tau):= e^{-\frac{1}{p-1}\tau} 
	u(\tilde x + e^{-\frac{1}{2}\tau} \eta, \tilde t-e^{-\tau})
	= (\tilde t-t)^\frac{1}{p-1} u(x,t), \\
	&\eta :=(\tilde t-t)^{-\frac{1}{2}}(x-\tilde x), 
	\quad \tau:=-\log(\tilde t-t). 
\end{aligned}
\end{equation}
We see that $w$ solves 
\begin{equation}\label{eq:weq}
	\left\{ 
	\begin{aligned}
	&w_\tau  = \Delta w - \frac{\eta}{2} \cdot \nabla w 
	- \frac{w}{p-1} + |w|^{p-1}w, 
	&&\eta\in \Omega(\tau), \;
	\tau\in  (\tilde \tau_0,\infty), \\
	&w=0, &&\eta\in \partial \Omega(\tau), \; 
	\tau\in  (\tilde \tau_0,\infty), \\
	&w(\eta,\tilde \tau_0)
	= \tilde t^\frac{1}{p-1} u_0(\tilde x + \tilde t^\frac{1}{2} \eta), 
	&& \eta\in \Omega(\tilde \tau_0), 
	\end{aligned}
	\right.
\end{equation}
where $\tilde \tau_0:=-\log \tilde t$ and 
\[
	\Omega(\tau)
	:=\{\eta\in \R^n; e^{-\frac{\tau}{2}} \eta +\tilde x\in \Omega \}
	= e^\frac{\tau}{2} (\Omega -\tilde x). 
\]
Here and below, we omit the subscripts $(\tilde x,\tilde t)$ 
if they are clear from the context.
Note that $w$ also satisfies, with $\rho(\eta):=e^{-|\eta|^2/4}$, 
\begin{equation}\label{eq:wrhoeq}
	w_\tau \rho 
	= \nabla \cdot (\rho \nabla w) -\frac{1}{p-1} w\rho  +|w|^{p-1} w\rho. 
\end{equation}
In the backward similarity variables, $E_{(\tilde x,\tilde t)}$ can be written as 
\begin{equation}\label{eq:EcErel}
	E_{(\tilde x,\tilde t)}(t)  
	= \cE_{(\tilde x,\tilde t)}(\tau), 
	\quad \tau=-\log(\tilde t-t), \; \tau > \tilde \tau_0, 
\end{equation}
where $\cE_{(\tilde x,\tilde t)}$ is defined by 
\begin{equation}\label{eq:cEdef}
	\cE_{(\tilde x,\tilde t)}(\tau):= 
	\int_{\Omega(\tau)} 
	\left( 
	\frac{ |\nabla w|^2}{2}
	-\frac{|w|^{p+1}}{p+1} 
	+\frac{w^2}{2(p-1)} \right) 
	\rho d\eta. 
\end{equation}

We examine the monotonicity of $\cE_{(\tilde x,\tilde t)}$. 
From Giga--Kohn \cite[Proposition 2.1]{GK87} 
(see also \eqref{eq:locmonoequality} below 
with $\psi_k$ replaced by $1$), it follows that 
\[
	\frac{d}{d\tau} \cE_{(\tilde x,\tilde t)}(\tau) 
	= - \int_{\Omega(\tau)} |w_\tau|^2 \rho d\eta 
	-\frac{1}{4} \int_{\partial \Omega(\tau)} (\eta \cdot \nu_\eta) 
	|\nabla w \cdot \nu_\eta|^2 \rho dS(\eta) 
\]
for $\tau > \tilde \tau_0$, 
where $\nu_\eta$ is the exterior unit normal 
at $\eta\in \partial \Omega(\tau)$ 
and $dS=dS(\eta)$ is the surface area element. 
Since the tangential derivatives of $w$ vanish by $w=0$ on $\partial \Omega(\tau)$, 
we have $\nabla w=(\nabla w\cdot \nu_\eta)\nu_\eta$ 
and $|\nabla w|^2=|\nabla w \cdot \nu_\eta|^2$ 
on $\partial\Omega(\tau)$.
Then,  
\begin{equation}\label{eq:GKoriginal}
\begin{aligned}
	\cE_{(\tilde x,\tilde t)}(\tau_2) 
	&= \cE_{(\tilde x,\tilde t)}(\tau_1) 
	- \int_{\tau_1}^{\tau_2} \int_{\Omega(\tau)} |w_\tau|^2 \rho d\eta d\tau \\
	&\quad 
	-\frac{1}{4} \int_{\tau_1}^{\tau_2}
	\int_{\partial \Omega(\tau)} (\eta \cdot \nu_\eta) 
	|\nabla w|^2 \rho dS(\eta) d\tau 
\end{aligned}
\end{equation}
for $\tilde \tau_0 < \tau_1 < \tau_2<\infty$. 
If the domain $\Omega$ is convex, then 
we can easily check the monotonicity of 
$\tau\mapsto \cE_{(\tilde x,\tilde t)}(\tau)$ 
by $\eta \cdot \nu_\eta\geq0$ on $\partial\Omega(\tau)$.  
However, that is not the case when $\Omega$ is possibly non-convex.
To overcome this difficulty, we proceed with the analysis 
based on Blatt--Struwe \cite[Section 4]{BS15b}.

Up to the proof of Proposition \ref{pro:qmono}, 
we temporarily assume $\tilde x\in\Omega$ (interior case). 
Scaling back to the original variables shows that 
\begin{equation}\label{eq:monoBdiff}
\begin{aligned}
	E_{(\tilde x,\tilde t)}(t_2)  
	\leq 
	E_{(\tilde x,\tilde t)}(t_1) 
	- \frac{1}{4} \int_{t_1}^{t_2} B_{(\tilde x,\tilde t)} (t) dt
\end{aligned}
\end{equation}
for $0 < t_1<t_2<\tilde t$ with 
$\tau_1=-\log(\tilde t-t_1)$ and $\tau_2=-\log(\tilde t-t_2)$, where 
\begin{equation}\label{eq:Bdef}
	B_{(\tilde x,\tilde t)} (t) := 
	(\tilde t-t)^{\frac{2}{p-1}} 
	\int_{\partial \Omega} 
	((x-\tilde x)\cdot \nu_x) |\nabla u|^2 
	K_{(\tilde x,\tilde t)}(x,t) dS(x)
\end{equation}
for $0<t<\tilde t$ and 
$\nu_x$ is the exterior unit normal at $x\in \partial \Omega$.

For estimating the boundary integral, we set  
\begin{equation}\label{eq:Bpmdef}
	B_{(\tilde x,\tilde t)}^\pm (t) := 
	(\tilde t-t)^{\frac{2}{p-1}} 
	\int_{\partial \Omega} 
	((x-\tilde x)\cdot \nu_x)_\pm |\nabla u|^2 
	K_{(\tilde x,\tilde t)}(x,t) dS(x)
\end{equation}
for $0<t<\tilde t$, 
where $((x-\tilde x)\cdot \nu_x)_\pm\geq 0$ are the positive and negative parts of 
$x\mapsto (x-\tilde x)\cdot \nu_x$, respectively. 
Note that $B_{(\tilde x,\tilde t)} 
= B_{(\tilde x,\tilde t)}^+ - B_{(\tilde x,\tilde t)}^-$. 
We will observe the relations between the Giga--Kohn energy 
and the boundary integrals. 
In the next lemma, we only consider the case $\tilde t<T$. 
The case $\tilde t=T$ will be handled 
in the proof of Proposition \ref{pro:qmono}.

\begin{lemma}
Let $\tilde x\in\Omega$, $T/2<\tilde t<T$ and $0<\delta<\sqrt{T/2}$. 
Then, 
\begin{align}
	&\label{eq:boudel1}
	\sup_{\tilde x'\in \Omega} 
	\int_{\tilde t-\delta^2}^{\tilde t} B_{(\tilde x',\tilde t)} (t) dt
	\leq 
	4\sup_{\tilde x'\in \Omega} 
	E_{(\tilde x',\tilde t)}(\tilde t-\delta^2), \\
	&\label{eq:boudel2}
	E_{(\tilde x,\tilde t)}(t)
	\leq 
	E_{(\tilde x,\tilde t)}(t') 
	+ \frac{1}{4} \sup_{\tilde x'\in\Omega}\int_{\tilde t-\delta^2}^{\tilde t} 
	B_{(\tilde x',\tilde t)}^- (t) dt, 
\end{align}
for $\tilde t-\delta^2 \leq t' < t < \tilde t$, 
where the right-hand side in \eqref{eq:boudel1} 
satisfies 
\begin{equation}\label{eq:boundEcalc}
\begin{aligned}
	4\sup_{\tilde x'\in \Omega} 
	E_{(\tilde x',\tilde t)}(\tilde t-\delta^2)
	&\leq 2 (4\pi)^{\frac{n}{2}} \delta^{\frac{2(p+1)}{p-1}} 
	\|\nabla u(\cdot,\tilde t-\delta^2)\|_{L^\infty(\Omega)}^2  \\
	&\quad 
	+ \frac{2}{p-1} (4\pi)^{\frac{n}{2}}\delta^{\frac{2(p+1)}{p-1}-2} 
	\|u(\cdot,\tilde t-\delta^2)\|_{L^\infty(\Omega)}^2. 
\end{aligned}
\end{equation}
\end{lemma}

\begin{proof}
Let $\tilde x\in \Omega$ and $T/2<\tilde t<T$. 
First, we prove \eqref{eq:boudel1} and \eqref{eq:boundEcalc}. 
By \eqref{eq:monoBdiff}, for $0 < t_1<t_2<\tilde t$, 
\[
	\int_{t_1}^{t_2} 
	B_{(\tilde x,\tilde t)} (t) dt
	\leq 
	4E_{(\tilde x,\tilde t)}(t_1) 
	- 4 E_{(\tilde x,\tilde t)}(t_2). 
\] 
Since $\tilde t<T$ is prior to the blow-up time $t=T$, 
the boundedness in \eqref{eq:regu} 
together with $\int_\Omega K_{(\tilde x,\tilde t)} dx\leq (4\pi)^{n/2}$ 
and \eqref{eq:GKenedef} implies that 
\[
\begin{aligned}
	-4E_{(\tilde x,\tilde t)}(t_2)
	&\leq 
	4 (\tilde t-t_2)^{\frac{p+1}{p-1}} 
	\int_\Omega 
	\frac{\|u(\cdot,t_2)\|_{L^\infty(\Omega)}^{p+1}}{p+1} 
	K_{(\tilde x,\tilde t)} dx \\
	&\leq 
	\frac{4 (4\pi)^{\frac{n}{2}}}{p+1} (\tilde t-t_2)^{\frac{p+1}{p-1}} 
	\|u(\cdot,t_2)\|_{L^\infty(\Omega)}^{p+1} 
	\to 0 
\end{aligned}
\]
as $t_2\to\tilde t$.
Then by taking $t_1=\tilde t-\delta^2$, we see that 
\begin{equation}\label{eq:Ebound42}
\begin{aligned}
	\int_{\tilde t-\delta^2}^{\tilde t} 
	B_{(\tilde x,\tilde t)} (t) dt
	&\leq 
	4E_{(\tilde x,\tilde t)}(\tilde t-\delta^2)  \\
	&\leq 
	2 (4\pi)^{\frac{n}{2}}\delta^{\frac{2(p+1)}{p-1}} 
	\|\nabla u(\cdot,\tilde t-\delta^2)\|_{L^\infty(\Omega)}^2 \\
	&\quad 
	+ \frac{2}{p-1}(4\pi)^{\frac{n}{2}}\delta^{\frac{2(p+1)}{p-1}-2} 
	\|u(\cdot,\tilde t-\delta^2)\|_{L^\infty(\Omega)}^2
\end{aligned}
\end{equation}
for $T/2<\tilde t<T$, 
where $\tilde t-\delta^2>0$, since $\tilde t>T/2$ and $\delta<\sqrt{T/2}$. 
Therefore, taking the supremum in $\tilde x$ 
yields \eqref{eq:boudel1} and \eqref{eq:boundEcalc}.

We next prove \eqref{eq:boudel2}. 
By \eqref{eq:monoBdiff} and $B_{(\tilde x,\tilde t)}^-\geq0$, 
\[
\begin{aligned}
	&E_{(\tilde x,\tilde t)}(t_2)
	\leq 
	E_{(\tilde x,\tilde t)}(t_1) 
	-  \frac{1}{4} \int_{t_1}^{t_2} 
	B_{(\tilde x,\tilde t)} (t) dt \\
	&\leq 
	E_{(\tilde x,\tilde t)}(t_1) 
	+  \frac{1}{4} \int_{t_1}^{t_2} 
	B_{(\tilde x,\tilde t)}^-  (t) dt
	\leq 
	E_{(\tilde x,\tilde t)}(t_1) 
	+  \frac{1}{4} \int_{t_1}^{\tilde t} 
	B_{(\tilde x,\tilde t)}^-  (t) dt
\end{aligned}
\]
for $0 < t_1<t_2<\tilde t$. 
In particular, for $\tilde t-\delta^2\leq t_1<t_2<\tilde t$, 
\[
	E_{(\tilde x,\tilde t)}(t_2)
	\leq 
	E_{(\tilde x,\tilde t)}(t_1) 
	+ \frac{1}{4} \int_{\tilde t-\delta^2}^{\tilde t} 
	B_{(\tilde x,\tilde t)}^-  (t) dt. 
\]
Then \eqref{eq:boudel2} follows. 
\end{proof}

Let $\tilde x\in\Omega$, $T/2<\tilde t<T$ and $0<\delta<\sqrt{T/2}$. 
For $r>0$, we write  
$B_r(\tilde x):= \{x\in \R^n; |x-\tilde x|<r\}$  and 
$\partial\Omega_r(\tilde x):= B_r(\tilde x) \cap \partial \Omega$. 
To handle $B_{(\tilde x,\tilde t)}$, we set  
\begin{equation}\label{eq:betardef}
	\beta(r) := 
	r^{\frac{2(p+1)}{p-1}-n-1} 
	\sup_{\substack{\tilde x\in \Omega, \\ 
	\dist(\tilde x,\partial\Omega)<r}}
	\int_{\tilde t-2r^2}^{\tilde t-r^2} 
	\int_{\partial \Omega_r(\tilde x)}
	|\nabla u(x,t)|^2 dS(x) dt
\end{equation}
for $0<r<\sqrt{\tilde t/2}$, 
where this quantity corresponds to 
Blatt--Struwe \cite[p. 2278]{BS15b}. 
We remark that the sign of $(2(p+1)/(p-1))-n-1$ 
has no relation to the following argument. 
Let us estimate $\beta$ by using $B^+_{(\tilde x,\tilde t)}$.

\begin{lemma}\label{lem:betabyB+}
There exist $0<\delta<\sqrt{T/2}$ depending only on $n$, $p$, $\Omega$ and $T$ 
and $C>0$ depending only on $n$, $p$ and $\Omega$ such that 
\[
\begin{aligned}
	\sup_{0<r<\delta/\sqrt{2}} \beta(r)
	&\leq 
	C \sup_{\tilde x\in \Omega} \int_{\tilde t-\delta^2}^{\tilde t} 
	B_{(\tilde x,\tilde t)}^+ (t) dt.  
\end{aligned}
\]
\end{lemma}

\begin{proof}
We follow \cite[Lemma 4.2]{BS15b} 
without using the cut-off function appearing there.
Since $\Omega$ is uniformly  $C^{2+\alpha}$ 
(in particular, uniformly  $C^2$), 
there exists $C_\Omega>0$ depending only on $\Omega$ such that 
\begin{equation}\label{eq:geom}
	(x-\tilde x) \cdot \nu_x 
	\geq \dist(\tilde x,\partial \Omega) - C_\Omega |x-\tilde x|^2
	\quad 
	\mbox{ for }x\in\partial \Omega, \; \tilde x\in \Omega. 
\end{equation}
See \cite[(4.1)]{BS15b} for the proof. In particular, 
\[
\begin{aligned}
	( (x-\tilde x) \cdot \nu_x)_+ 
	&\geq 
	\dist(\tilde x,\partial \Omega) - C_\Omega |x-\tilde x|^2 
	+ ( (x-\tilde x) \cdot \nu_x)_-  \\
	&\geq \dist(\tilde x,\partial \Omega) - C_\Omega |x-\tilde x|^2. 
\end{aligned}
\]

Let $0<r<\delta/\sqrt{2}$ and let 
$\tilde x_0\in \Omega$ satisfy $\dist(\tilde x_0, \partial\Omega) < r$.  
The above geometric observation together with \eqref{eq:Bpmdef} 
and shrinking the domain of integration shows that, 
for $\tilde x_1\in \Omega$ with 
$|\tilde x_1-\tilde x_0|\leq r$ and $\dist(\tilde x_1,\partial\Omega)=r$, 
\[
\begin{aligned}
	&\sup_{\tilde x\in \Omega} \int_{\tilde t-\delta^2}^{\tilde t} 
	B_{(\tilde x,\tilde t)}^+ (t) dt 
	\geq 
	\int_{\tilde t-\delta^2}^{\tilde t} 
	B_{(\tilde x_1,\tilde t)}^+ (t) dt  \\
	&\geq 
	\int_{\tilde t-2 r^2}^{\tilde t-r^2} 
	(\tilde t-t)^{\frac{2}{p-1}} 
	\int_{\partial \Omega_r(\tilde x_0)} 
	((x-\tilde x_1)\cdot \nu_x)_+ |\nabla u|^2 
	K_{(\tilde x_1,\tilde t)} dS(x) dt\\
	&\geq 
	\frac{1}{C} r^{\frac{4}{p-1}-n}
	\Bigg( 
	\int_{\tilde t-2r^2}^{\tilde t-r^2}  
	\int_{\partial \Omega_r(\tilde x_0)} 
	\dist(\tilde x_1,\partial \Omega) |\nabla u|^2 
	e^{-\frac{|x-\tilde x_1|^2}{4(\tilde t-t)}} 
	dS(x) dt \\
	&\quad 
	- C_\Omega \int_{\tilde t-2r^2}^{\tilde t-r^2}  
	\int_{\partial \Omega_r(\tilde x_0)} 
	\left( |x-\tilde x_1|^2 e^{-\frac{|x-\tilde x_1|^2}{4(\tilde t-t)}}
	\right) |\nabla u|^2   dS(x) dt \Bigg). 
\end{aligned}
\]
In the right-hand side, for the first integral, 
we use $\dist(\tilde x_1,\partial\Omega) = r$ and 
$|x-\tilde x_1|^2 \leq 4 r^2$ for $x\in \partial \Omega_r(\tilde x_0)$. 
For the second integral, we apply a fundamental inequality 
$\sup_{z\in \R^n} |z|^2 e^{-|z|^2/(4(\tilde t-t))}
\leq C'(\tilde t-t)$ with an absolute constant $C'>0$. 
Then, we see that 
\[
	\sup_{\tilde x\in \Omega} \int_{\tilde t-\delta^2}^{\tilde t} 
	B_{(\tilde x,\tilde t)}^+ (t) dt  
	\geq 
	r^{\frac{4}{p-1}-n+1}
	\left( \frac{1}{C''} - C'' r \right) 
	\int_{\tilde t-2r^2}^{\tilde t-r^2}  
	\int_{\partial \Omega_r(\tilde x_0)} 
	|\nabla u|^2 
	dS dt, 
\]
where $C''>1$ depends only on $n$, $p$ and $\Omega$ 
and the right-hand side is independent of $\tilde x_1$. 
Hence by $4/(p-1)=(2(p+1)/(p-1))-2$ 
and by choosing $0<\delta<\sqrt{T/2}$ so small that 
$(1/C'') - C'' r \geq  1/2C''$ holds for $0<r<\delta/\sqrt{2}$, 
\[
	r^{\frac{2(p+1)}{p-1}-n-1} 
	\int_{\tilde t-2r^2}^{\tilde t-r^2}  
	\int_{\partial \Omega_r(\tilde x_0)}  |\nabla u|^2 dS dt
	\leq 
	2C'' \sup_{\tilde x\in \Omega} \int_{\tilde t-\delta^2}^{\tilde t} 
	B_{(\tilde x,\tilde t)}^+ (t) dt 
\]
for $0<r<\delta/\sqrt{2}$ and 
$\tilde x_0\in \Omega$ with $\dist(\tilde x_0, \partial\Omega) < r$. 
Since the right-hand side is independent of $\tilde x_0$ and $r$, 
the lemma follows from \eqref{eq:betardef}. 
\end{proof}

We have estimated $\beta$ from above by using $B_{(\tilde x_0,\tilde t)}^+$. 
On the other hand, we estimate 
$\beta$ from below by using $B_{(\tilde x_0,\tilde t)}^-$.

\begin{lemma}\label{lem:B-bybeta}
Let $0<\delta<\sqrt{T/2}$ be as in Lemma \ref{lem:betabyB+}. 
Then there exists $C>0$ depending only on $n$, $p$ and $\Omega$ such that 
\[
	\sup_{\tilde x\in \Omega}\int_{\tilde t-\delta^2}^{\tilde t} 
	B_{(\tilde x,\tilde t)}^- (t) dt 
	\leq 
	C \delta \sup_{0<r<\delta/\sqrt{2}} \beta(r). 
\]
\end{lemma}

\begin{proof}
We give a detailed version 
of \cite[Lemma 4.1]{BS15b} adapted to our setting.
By \eqref{eq:geom}, if $( (x-\tilde x) \cdot \nu_x)_- >0$, then 
$( (x-\tilde x) \cdot \nu_x)_+ =0$ and  
\[
	( (x-\tilde x) \cdot \nu_x)_- 
	\leq 
	- \dist(\tilde x,\partial \Omega) + C_\Omega |x-\tilde x|^2 
	+ ( (x-\tilde x) \cdot \nu_x)_+ 
	\leq C_\Omega |x-\tilde x|^2. 
\]
Thus, we have 
\begin{equation}\label{eq:bminus}
	( (x-\tilde x) \cdot \nu_x)_- \leq C_\Omega |x-\tilde x|^2 
	\quad 
	\mbox{ for }x\in\partial \Omega, \; \tilde x\in \Omega. 
\end{equation}
Let $\tilde x\in \Omega$ and let 
$0<\delta<\sqrt{T/2}$ be as in Lemma \ref{lem:betabyB+}. 
Set $\delta_i:=2^{-(i-1)/2}\delta$ for $i\geq 1$. 
Then, \eqref{eq:bminus} implies that 
\[
\begin{aligned}
	&\int_{\tilde t-\delta^2}^{\tilde t} 
	B_{(\tilde x,\tilde t)}^- (t) dt  \\
	&\leq 
	C_\Omega \int_{\tilde t-\delta^2}^{\tilde t} 
	(\tilde t-t)^{\frac{2}{p-1}-\frac{n}{2}} 
	\int_{\partial \Omega} 
	\left( |x-\tilde x|^2 
	e^{-\frac{|x-\tilde x|^2}{8(\tilde t-t)}} \right) |\nabla u|^2 
	e^{-\frac{|x-\tilde x|^2}{8(\tilde t-t)}} 
	dS dt \\
	&\leq 
	C \int_{\tilde t-\delta^2}^{\tilde t} 
	(\tilde t-t)^{\frac{2}{p-1}-\frac{n}{2}+1} 
	\int_{\partial \Omega} 
	|\nabla u|^2 e^{-\frac{|x-\tilde x|^2}{8(\tilde t-t)}} dS dt. 
\end{aligned}
\]

Let us decompose $\partial\Omega\times (\tilde t-\delta^2,\tilde t)$ into 
\[
	\partial\Omega \times (\tilde t-\delta^2,\tilde t) 
	= \bigcup_{i=1}^\infty 
	\left( \partial\Omega \times (\tilde t-\delta_i^2, \tilde t-\delta_{i+1}^2] 
	\right). 
\]
We write 
$\partial A_r(\tilde x):= \partial\Omega_{2r}(\tilde x) 
\setminus \partial\Omega_r(\tilde x)$ for $r>0$, 
where $\partial\Omega_r(\tilde x)= B_r(\tilde x) \cap \partial \Omega$. 
Then, for each $i\geq1$, we also decompose 
$\partial\Omega\times (\tilde t-\delta_i^2, \tilde t-\delta_{i+1}^2]$ into 
\[
\begin{aligned}
	&\partial\Omega \times (\tilde t-\delta_i^2, \tilde t-\delta_{i+1}^2] \\
	&= 
	\left( \partial\Omega_{\delta_{i+2}}(\tilde x)
	\cup \left( \bigcup_{j=0}^\infty 
	\partial\Omega_{2\times 2^j\delta_{i+2}}(\tilde x)
	\setminus \partial\Omega_{2^j \delta_{i+2}}(\tilde x) \right) \right) 
	\times (\tilde t-\delta_i^2, \tilde t-\delta_{i+1}^2] \\
	&= 
	\left( \partial\Omega_{\delta_{i+2}}(\tilde x) 
	\times (\tilde t-\delta_i^2, \tilde t-\delta_{i+1}^2] \right)
	\cup \left( \bigcup_{j=0}^\infty 
	\partial A_{2^j\delta_{i+2}}(\tilde x) 
	\times (\tilde t-\delta_i^2, \tilde t-\delta_{i+1}^2] \right). 
\end{aligned}
\]
From these decompositions, 
it follows that 
\[
\begin{aligned}
	&\int_{\tilde t-\delta^2}^{\tilde t} 
	(\tilde t-t)^{\frac{2}{p-1}-\frac{n}{2}+1} 
	\int_{\partial \Omega} 
	|\nabla u|^2 e^{-\frac{|x-\tilde x|^2}{8(\tilde t-t)}} dS dt \\
	&= 
	\sum_{i=1}^\infty 
	\int_{\tilde t-\delta_i^2}^{\tilde t-\delta_{i+1}^2} 
	(\tilde t-t)^{\frac{2}{p-1}-\frac{n}{2}+1} 
	\int_{\partial \Omega_{\delta_{i+2}}(\tilde x)} 
	|\nabla u|^2 e^{-\frac{|x-\tilde x|^2}{8(\tilde t-t)}} dS dt \\
	&\quad + 
	\sum_{i=1}^\infty 
	\int_{\tilde t-\delta_i^2}^{\tilde t-\delta_{i+1}^2} 
	(\tilde t-t)^{\frac{2}{p-1}-\frac{n}{2}+1} 
	\sum_{j=0}^\infty
	\int_{\partial A_{2^j \delta_{i+2}}(\tilde x)} 
	|\nabla u|^2 e^{-\frac{|x-\tilde x|^2}{8(\tilde t-t)}} dS dt \\
	&=: \sum_{i=1}^\infty I_i
	+ \sum_{i=1}^\infty \sum_{j=0}^\infty J_{i,j}. 
\end{aligned}
\]

Let us consider $I_i$. 
By $\delta_{i+2}\leq \delta_{i+1}$ and 
$\delta_i=\sqrt{2}\delta_{i+1}$, 
\[
\begin{aligned}
	&I_i \leq 
	(\delta_i)^{\frac{4}{p-1}+2} (\delta_{i+1})^{-n} 
	\int_{\tilde t-\delta_i^2}^{\tilde t-\delta_{i+1}^2} 
	\int_{\partial \Omega_{\delta_{i+1}}(\tilde x)} |\nabla u|^2  dS dt \\
	&\leq 
	C \delta_{i+1} (\delta_{i+1})^{\frac{2(p+1)}{p-1}-n-1} 
	\int_{\tilde t-2\delta_{i+1}^2}^{\tilde t - \delta_{i+1}^2} 
	\int_{\partial \Omega_{\delta_{i+1}}(\tilde x)} |\nabla u|^2  dS dt 
\end{aligned}
\]
for all $\tilde x\in\Omega$.  We observe that 
if $\dist(\tilde x,\partial \Omega)\geq \delta_{i+1}$, 
then $\partial \Omega_{\delta_{i+1}}(\tilde x)=
B_{\delta_{i+1}}(\tilde x) \cap \partial\Omega = \emptyset$ and $I_i=0$. 
Thus, the definition of $\beta$ in \eqref{eq:betardef} yields 
\[
\begin{aligned}
	I_i \leq 
	\sup_{\tilde x\in \Omega} I_i 
	= 
	\sup_{\substack{\tilde x\in \Omega, \\
	\dist(\tilde x,\partial\Omega)<\delta_{i+1}}} I_i 
	\leq 
	C \delta_{i+1} \beta(\delta_{i+1}). 
\end{aligned}
\]
Then by $\delta_{i+1}\leq \delta/\sqrt{2}$ and 
$\delta_{i+1}=2^{-i/2}\delta$, 
\[
\begin{aligned}
	\sum_{i=1}^\infty I_i &\leq 
	C \sup_{0<r<\delta/\sqrt{2}} \beta(r) 
	\sum_{i=1}^\infty 2^{-\frac{i}{2}} \delta 
	\leq 
	C \delta  \sup_{0<r<\delta/\sqrt{2}} \beta(r). 
\end{aligned}
\]

As for $J_{i,j}$, 
we note that for 
$x\in \partial A_{2^j \delta_{i+2}}(\tilde x)$
and $\tilde t-\delta_i^2 < t < \tilde t-\delta_{i+1}^2$, 
\[
	e^{-\frac{|x-\tilde x|^2}{8(\tilde t-t)}}
	\leq 
	e^{-\frac{(2^j \delta_{i+2})^2}{8\delta_i^2}} 
	= 
	e^{-\frac{2^{2j-1} \delta_{i+1}^2}{16\delta_{i+1}^2}} 
	= e^{-2^{2j-5} }. 
\]
Thus, 
\[
\begin{aligned}
	J_{i,j}&\leq 
	(\delta_i)^{\frac{4}{p-1}+2} (\delta_{i+1})^{-n}  e^{-2^{2j-5} }
	\int_{\tilde t-\delta_i^2}^{\tilde t-\delta_{i+1}^2} 
	\int_{\partial A_{2^j \delta_{i+2}}(\tilde x)} 
	|\nabla u|^2 dS dt \\
	&\leq 
	C\delta_{i+1} e^{-2^{2j-5} }
	(\delta_{i+1})^{\frac{2(p+1)}{p-1}-n-1} 
	\int_{\tilde t-2\delta_{i+1}^2}^{\tilde t-\delta_{i+1}^2} 
	\int_{\partial A_{2^j \delta_{i+2}}(\tilde x)} 
	|\nabla u|^2 dS dt. 
\end{aligned}
\]
We consider a covering of $\partial A_{2^j \delta_{i+2}}(\tilde x)$ 
by using balls of radius $\delta_{i+2}$. 
Observe that 
$\partial A_{2^j \delta_{i+2}}(\tilde x) 
\subset B_{2^{j+1} \delta_{i+2}}(\tilde x)$ 
and that $B_{2^{j+1} \delta_{i+2}}(\tilde x)$ can be covered by 
$C_n 2^{n(j+1)}$ balls 
of radius $\delta_{i+2}$, where $C_n>0$ is a constant depending only on $n$. 
Therefore, $\partial A_{2^j \delta_{i+2}}(\tilde x)$ 
also can be covered by a family of balls 
$\{ B_{\delta_{i+2}}(x_{i,j}^{(l)}) \}_{l=1}^{N_j}$, 
where $x_{i,j}^{(l)}\in \partial \Omega$, $N_j \leq 2^{nj} C_n'$ 
and $C_n'>0$ depends only on $n$. 
Then, we see that 
\[
	J_{i,j}\leq 
	C\delta_{i+1} e^{-2^{2j-5} }
	\sum_{l=1}^{N_j} 
	(\delta_{i+1})^{\frac{2(p+1)}{p-1}-n-1} 
	\int_{\tilde t-2\delta_{i+1}^2}^{\tilde t-\delta_{i+1}^2} 
	\int_{
	B_{\delta_{i+2}}(x_{i,j}^{(l)}) \cap \partial \Omega
	 }
	|\nabla u|^2 dS dt. 
\]
For each $x_{i,j}^{(l)}\in \partial \Omega$, we choose 
$\tilde x_{i,j}^{(l)}\in \Omega$ such that 
\[
	\dist(\tilde x_{i,j}^{(l)}, \partial \Omega)  
	= |x_{i,j}^{(l)} - \tilde x_{i,j}^{(l)}|
	\leq  \frac{\sqrt{2}-1}{\sqrt{2}} \delta_{i+1}. 
\]
We note that this choice guarantees that 
$B_{\delta_{i+2}}(x_{i,j}^{(l)}) 
\subset B_{\delta_{i+1}}(\tilde x_{i,j}^{(l)})$. 
Then, 
\[
\begin{aligned}
	J_{i,j}
	&\leq 
	C\delta_{i+1} e^{-2^{2j-5} }
	\sum_{l=1}^{N_j} 
	(\delta_{i+1})^{\frac{2(p+1)}{p-1}-n-1} 
	\int_{\tilde t-2\delta_{i+1}^2}^{\tilde t-\delta_{i+1}^2} 
	\int_{
	B_{\delta_{i+1}}(\tilde x_{i,j}^{(l)}) \cap \partial \Omega
	} 
	|\nabla u|^2 dS dt \\
	&\leq 
	C\delta_{i+1} e^{-2^{2j-5} }
	\sum_{l=1}^{N_j} 
	\beta(\delta_{i+1}) 
	\leq 
	C\delta_{i+1} e^{-2^{2j-5} }
	2^{n j}
	\beta(\delta_{i+1}), 
\end{aligned}
\]
and so 
\[
\begin{aligned}
	\sum_{i=1}^\infty \sum_{j=0}^\infty J_{i,j}
	&\leq 
	C \sum_{i=1}^\infty \sum_{j=0}^\infty 
	\delta_{i+1} e^{-2^{2j-5} }
	2^{n j}
	\beta(\delta_{i+1}) \\
	&\leq 
	C \sum_{i=1}^\infty \delta_{i+1} \beta(\delta_{i+1}) 
	\leq C \delta  \sup_{0<r<\delta/\sqrt{2}} \beta(r). 
\end{aligned}
\]

From the above estimates, it follows that 
\[
	\int_{\tilde t-\delta^2}^{\tilde t} 
	B_{(\tilde x,\tilde t)}^- (t) dt 
	\leq 
	C \delta  \sup_{0<r<\delta/\sqrt{2}} \beta(r) 
	\quad 
	\mbox{ for all }\tilde x\in\Omega, 
\]
where  
$C>0$ depends only on $n$, $p$ and $\Omega$. 
Since the right-hand side is independent of 
$\tilde x$, 
taking the supremum in $\tilde x$ 
yields the desired inequality. 
\end{proof}

We are now in a position to prove Proposition \ref{pro:qmono}. 

\begin{proof}[Proof of Proposition \ref{pro:qmono}]
We first consider the interior case $\tilde x\in\Omega$. 
Let $T/2<\tilde t<T$ and let 
$0<\delta<\sqrt{T/2}$ be as in Lemma \ref{lem:betabyB+}. 
By \eqref{eq:boudel2} and Lemma \ref{lem:B-bybeta}, we have 
\begin{equation}\label{eq:betaBsupE111}
\begin{aligned}
	E_{(\tilde x,\tilde t)}(t)
	&\leq 
	E_{(\tilde x,\tilde t)}(t') 
	+ \frac{1}{4} \sup_{\tilde x'\in\Omega}
	\int_{\tilde t-\delta^2}^{\tilde t} 
	B_{(\tilde x',\tilde t)}^- (t) dt \\
	&\leq 
	E_{(\tilde x,\tilde t)}(t') 
	+ C \delta \sup_{0<r<\delta/\sqrt{2}} \beta(r) 
\end{aligned}
\end{equation}
for $\tilde t-\delta^2 \leq t' < t< \tilde t$, 
where $C>0$ depends only on $n$, $p$ and $\Omega$. 
Then, it suffices to estimate $\sup_r \beta(r)$. 
From Lemma \ref{lem:betabyB+}, 
$B_{(\tilde x,\tilde t)}^+ 
= B_{(\tilde x,\tilde t)} + B_{(\tilde x,\tilde t)}^-$ 
and Lemma \ref{lem:B-bybeta}, it follows that 
\[
\begin{aligned}
	\sup_{0<r<\delta/\sqrt{2}} \beta(r)
	&\leq 
	C \sup_{\tilde x\in \Omega} \int_{\tilde t-\delta^2}^{\tilde t} 
	B_{(\tilde x,\tilde t)}^+ (t) dt \\
	&\leq 
	C \sup_{\tilde x\in \Omega} \int_{\tilde t-\delta^2}^{\tilde t} 
	B_{(\tilde x,\tilde t)} (t) dt
	+ C \sup_{\tilde x\in \Omega} \int_{\tilde t-\delta^2}^{\tilde t} 
	B_{(\tilde x,\tilde t)}^- (t) dt \\
	&\leq 
	C \sup_{\tilde x\in \Omega} \int_{\tilde t-\delta^2}^{\tilde t} 
	B_{(\tilde x,\tilde t)} (t) dt
	+ C \delta \sup_{0<r<\delta/\sqrt{2}} \beta(r). 
\end{aligned}
\]
By choosing $0<\delta<\sqrt{T/2}$ sufficiently small 
depending only on $n$, $p$, $\Omega$ and $T$,  
and then by \eqref{eq:boudel1}, 
\begin{equation}\label{eq:betaBsupE}
\begin{aligned}
	\sup_{0<r<\delta/\sqrt{2}} \beta(r)
	\leq 
	C \sup_{\tilde x\in \Omega} \int_{\tilde t-\delta^2}^{\tilde t} 
	B_{(\tilde x,\tilde t)} (t) dt 
	\leq 
	C \sup_{\tilde x\in \Omega} E_{(\tilde x,\tilde t)}(\tilde t-\delta^2). 
\end{aligned}
\end{equation}
This together with \eqref{eq:betaBsupE111}
shows that for $T/2<\tilde t<T$ and  
$\tilde t-\delta^2 \leq t' < t< \tilde t$, 
\[
	E_{(\tilde x,\tilde t)}(t) 
	\leq 
	E_{(\tilde x,\tilde t)}(t') 
	+ C \delta \sup_{\tilde x\in \Omega} 
	E_{(\tilde x,\tilde t)}(\tilde t-\delta^2). 
\]

By the bound of 
$E_{(\tilde x,\tilde t)}(\tilde t-\delta^2)$ in \eqref{eq:Ebound42} 
and by $T/2<\tilde t<T$, 
\begin{equation}\label{eq:tiltbddEpr}
\begin{aligned}
	\sup_{\tilde x \in \Omega} 
	E_{(\tilde x,\tilde t)}(\tilde t-\delta^2)  
	&\leq 
	\frac{1}{2} (4\pi)^{\frac{n}{2}}\delta^{\frac{2(p+1)}{p-1}} 
	\|\nabla u(\cdot,\tilde t-\delta^2)\|_{L^\infty(\Omega)}^2 \\
	& \quad
	+ \frac{1}{2(p-1)}(4\pi)^{\frac{n}{2}}\delta^{\frac{2(p+1)}{p-1}-2} 
	\|u(\cdot,\tilde t-\delta^2)\|_{L^\infty(\Omega)}^2\\
	&\leq 
	\frac{1}{2} (4\pi)^{\frac{n}{2}}\delta^{\frac{2(p+1)}{p-1}} 
	\|\nabla u\|_{L^\infty(\Omega\times (\frac{T}{2}-\delta^2, T-\delta^2))}^2  \\
	& \quad
	+ \frac{1}{2(p-1)}(4\pi)^{\frac{n}{2}}\delta^{\frac{2(p+1)}{p-1}-2} 
	\|u\|_{L^\infty(\Omega\times (\frac{T}{2}-\delta^2, T-\delta^2))}^2, 
\end{aligned}
\end{equation}
where the right-hand side is independent of $\tilde t$. 
Thus, the estimate 
\[
\begin{aligned}
	E_{(\tilde x,\tilde t)}(t) 
	&\leq 
	E_{(\tilde x,\tilde t)}(t')  
	+ \tilde C \delta^{\frac{2(p+1)}{p-1}+1}  
	\|\nabla u\|_{L^\infty(\Omega\times (\frac{T}{2}-\delta^2, T-\delta^2))}^2 \\
	&\quad + \tilde C \delta^{\frac{2(p+1)}{p-1}-1} 
	\|u\|_{L^\infty(\Omega\times (\frac{T}{2}-\delta^2, T-\delta^2))}^2 
\end{aligned}
\]
holds for $T/2<\tilde t<T$ and  
$\tilde t-\delta^2 \leq t' < t< \tilde t$, 
where $\tilde C>0$ depends only on $n$, $p$ and $\Omega$. 
To take the limit $\tilde t\to T$, 
we set $T-\delta^2<t_0<T$ and $T-\delta^2\leq t'<t<t_0$. 
Then by the Lebesgue dominated convergence theorem 
with the help of the boundedness in \eqref{eq:regu} and 
$e^{-|x-\tilde x|^2/(4(\tilde t-t))} 
\leq e^{-|x-\tilde x|^2/(4T)}$, 
we see that 
$\lim_{\tilde t \to T}E_{(\tilde x,\tilde t)}(t)
= E_{(\tilde x,T)}(t)$. 
Thus, the desired estimate \eqref{eq:qmonoes} 
holds for all 
$T-\delta^2\leq t'<t<t_0<T$. 
Since $T-\delta^2<t_0<T$ is arbitrary, 
\eqref{eq:qmonoes} also holds 
for all $T-\delta^2\leq t'<t<T$ 
in the interior case $\tilde x\in \Omega$.

As for the boundary case $\tilde x\in \partial \Omega$, 
let $\tilde x'\in B_1(\tilde x)\cap \Omega$. 
We have 
\begin{equation}\label{eq:qmonotilxdash}
\begin{aligned}
	E_{(\tilde x',T)}(t) 
	&\leq 
	E_{(\tilde x',T)}(t')  
	+ C \delta^{\frac{2(p+1)}{p-1}+1}  
	\|\nabla u\|_{L^\infty(\Omega\times (\frac{T}{2}-\delta^2, T-\delta^2))}^2 \\
	&\quad + C \delta^{\frac{2(p+1)}{p-1}-1} 
	\|u\|_{L^\infty(\Omega\times (\frac{T}{2}-\delta^2, T-\delta^2))}^2 
\end{aligned}
\end{equation}
for $T-\delta^2 \leq t' < t < T$, 
where the second and third terms in the right-hand side 
are independent of the choice of $\tilde x'$. 
Note that for $\tilde x'\in B_1(\tilde x)\cap \Omega$, 
\[
	e^{-\frac{|x-\tilde x'|^2}{4(\tilde t-t)}}
	\leq e^{-\frac{|x-\tilde x'|^2}{4T}} 
	\leq e^{-\frac{|x-\tilde x|^2}{8T}}
	e^{\frac{|\tilde x-\tilde x'|^2}{4T}} 
	\leq 
	e^{\frac{1}{4T}} 
	e^{-\frac{|x-\tilde x|^2}{8T}}. 
\]
Then by the Lebesgue dominated convergence theorem, 
letting $\tilde x'\to \tilde x$ in \eqref{eq:qmonotilxdash} 
shows \eqref{eq:qmonoes} 
in the boundary case $\tilde x\in \partial \Omega$. 
The proof is complete. 
\end{proof}

\begin{remark}
By slightly modifying the above argument, 
we can change the center of scaling in time. 
Namely, the following can be proved for all $p>1$:  
There exists $C>0$ depending only on $n$, $p$ and $\Omega$ such that 
\[
\begin{aligned}
	E_{(\tilde x,\tilde t)}(t) 
	&\leq 
	E_{(\tilde x,\tilde t)}(t')  
	+ C \delta^{\frac{2(p+1)}{p-1}+1}  
	\|\nabla u\|_{L^\infty(\Omega\times (\frac{T}{2}-\delta^2, T-\delta^2))}^2 \\
	&\quad + C \delta^{\frac{2(p+1)}{p-1}-1} 
	\|u\|_{L^\infty(\Omega\times (\frac{T}{2}-\delta^2, T-\delta^2))}^2 
\end{aligned}
\]
for all $\tilde x\in \overline{\Omega}$, 
$T/2 < \tilde t \leq T$, 
$\tilde t-\delta^2 \leq t' < t < \tilde t$, 
where $0<\delta<\sqrt{T/2}$ depends only on $n$, $p$, $\Omega$ and $T$. 
This kind of estimate is expected to be useful in the characterization of blow-up. 
\end{remark}

\section{Estimates via quasi-monotonicity}\label{sec:esti}
In this section, we give several estimates 
by using the quasi-monotonicity formula. 
We remark that all the results 
in this section are valid for all $p>1$. 
For later use, we rewrite Proposition \ref{pro:qmono} 
in the backward similarity variables. 
To do so, we fix $0<\delta<1$ and $\tilde C>0$ in Proposition \ref{pro:qmono} 
and we set 
\begin{equation}\label{eq:stars}
\left\{ 
\begin{aligned}
	&\tau_*:= -\log(\delta^2), \\
	&\begin{aligned}
	C_*&:= \tilde C \delta^{\frac{2(p+1)}{p-1}+1}  
	\|\nabla u\|_{L^\infty(\Omega\times (\frac{T}{2}-\delta^2, T-\delta^2))}^2 \\
	&\quad + \tilde C \delta^{\frac{2(p+1)}{p-1}-1} 
	\|u\|_{L^\infty(\Omega\times (\frac{T}{2}-\delta^2, T-\delta^2))}^2, 
	\end{aligned}\\
	&\begin{aligned}
	C_*'&:= \frac{1}{2} (4\pi)^{\frac{n}{2}} \delta^{\frac{2(p+1)}{p-1}} 
	\|\nabla u(\cdot,T-\delta^2)\|_{L^\infty(\Omega)}^2   \\
	&\quad 
	+ \frac{1}{2(p-1)} (4\pi)^{\frac{n}{2}} \delta^{\frac{2(p+1)}{p-1}-2} 
	\|u(\cdot,T-\delta^2)\|_{L^\infty(\Omega)}^2, 
	\end{aligned}
\end{aligned}
\right.
\end{equation}
where these constants are independent of 
$\tilde x\in\overline{\Omega}$. 
Then, the following holds:

\begin{proposition}\label{pro:qm}
Let $\tau_*$ and $C_*$ be the constants given in \eqref{eq:stars}. 
Then, 
\[
	\cE_{(\tilde x,T)}(\tau)
	\leq 
	\cE_{(\tilde x,T)}(\tau') 
	+ C_* 
	\quad 
	\mbox{ for } \tilde x\in \overline{\Omega}, \; 
	\tau_* \leq \tau' < \tau < \infty, 
\]
where $\tau_*$ and $C_*$ are  independent 
of $\tilde x$. 
\end{proposition}

\begin{proof}
This immediately follows from Proposition \ref{pro:qmono} and \eqref{eq:EcErel}. 
\end{proof}

As the first application of Proposition \ref{pro:qm}, 
we give the following 
upper and lower bounds of $\cE_{(\tilde x,T)}$:

\begin{lemma}\label{lem:GKul}
Let $\tau_*$, $C_*$ and $C_*'$ be the constants given in \eqref{eq:stars}. 
Then, 
\[
	- C_* 
	\leq 
	\cE_{(\tilde x,T)}(\tau)
	\leq 
	C_* + C_*' 
	\quad 
	\mbox{ for } \tilde x\in \overline{\Omega},  \; 
	\tau_* \leq \tau' < \tau < \infty, 
\]
where $\tau_*$, $C_*$ and $C_*'$ are independent of 
$\tilde x$. 
\end{lemma}

\begin{proof}
By Proposition \ref{pro:qm}, we have 
$\cE_{(\tilde x,T)}(\tau) 
\leq \cE_{(\tilde x,T)}(\tau_*) + C_*$. 
Since $\cE_{(\tilde x,T)}(\tau_*)
= E_{(\tilde x,T)}(T-\delta^2)$ 
by \eqref{eq:EcErel}, 
we see from the same computation as in \eqref{eq:boundEcalc} that 
\[
	\sup_{\tilde x\in \overline{\Omega}}\cE_{(\tilde x,T)}(\tau_*) 
	= \sup_{\tilde x\in \Omega} E_{(\tilde x,T)}(T-\delta^2)
	\leq C_*'.  
\]
Hence we obtain the desired upper bound.

As for the lower bound, 
by Giga--Kohn \cite[Proposition 2.1]{GK87} 
(see also the derivation of \eqref{eq:locwp1eq} below
with $\psi_k$ replaced by $1$), 
\begin{equation}\label{eq:w2GK}
	\frac{1}{2} \frac{d}{d\tau} 
	\int_{\Omega(\tau)} |w|^2 \rho d\eta 
	= -2 \cE_{(\tilde x,T)}(\tau) 
	+ \frac{p-1}{p+1} \int_{\Omega(\tau)} |w|^{p+1} \rho d\eta
\end{equation}
for $\tilde x\in \overline{\Omega}$ and $\tau_* \leq \tau' < \tau < \infty$, 
where $w=w_{(\tilde x,T)}$.  We define 
\begin{equation}\label{eq:Wksdef}
	W(\tau):=
	\left(\int_{\Omega(\tau)} 
	|w_{(\tilde x,T)}(\eta,\tau)|^2 \rho(\eta) d\eta\right)^{1/2}. 
\end{equation}
Integrating \eqref{eq:w2GK} and using Proposition \ref{pro:qm} yield 
\[
\begin{aligned}
	\int_{\tau'}^\tau \int_{\Omega(\sigma)} 
	|w|^{p+1} \rho d\eta d\sigma 
	&= \frac{C_p}{4} (W^2(\tau)-W^2(\tau')) 
	+ C_p \int_{\tau'}^\tau \cE_{(\tilde x,T)}(\sigma) d\sigma  \\
	&\leq 
	\frac{C_p}{4} W^2(\tau)
	+ C_p (\cE_{(\tilde x,T)}(\tau')+C_*)(\tau-\tau')
\end{aligned}
\]
for any $\tau_*\leq \tau' < \tau < \infty$, 
where $C_p:=2(p+1)/(p-1)$.

Let us prove that 
$\cE_{(\tilde x,T)}(\tau')+C_* \geq 0$ 
for any $\tau_*\leq \tau' < \infty$. 
If not, there exists $\tau_*\leq \tau'_* < \infty$ such that 
$\cE_{(\tilde x,T)}(\tau'_*)+C_* < 0$. Then, 
\[
	\int_{\tau'_*}^\tau \int_{\Omega(\sigma)} |w|^{p+1} \rho d\eta d\sigma 
	\leq 
	\frac{C_p}{4} W^2(\tau)
	+ C_p (\cE_{(\tilde x,T)}(\tau'_*)+C_*)(\tau-\tau'_*)
	\leq \frac{C_p}{4} W^2(\tau)
\]
for $\tau'_* < \tau < \infty$. 
Jensen's inequality 
yields the existence of $C>0$ depending only on $n$ and $p$ such that 
\[
	\int_{\tau_*'}^\tau 
	\left( \int_{\Omega(\sigma)} |w|^2 \rho d\eta \right)^\frac{p+1}{2} d\sigma
	\leq 
	C \int_{\Omega(\tau)} |w|^2 \rho   d\eta  
	\quad \mbox{ for }\tau'_* < \tau < \infty, 
\]
contrary to 
the global-in-time existence of $w$. 
Hence $\cE_{(\tilde x,T)}(\tau')+C_* \geq 0$ 
for any $\tau_*\leq \tau' < \infty$. 
This gives the desired lower bound. 
\end{proof}

We next give bounds for the space-time integral of $|w_\tau|^2 \rho$ 
and the space integral of $|w|^2 \rho$. 
These bounds will be used several times in the next section.

\begin{lemma}
There exists $C>0$ depending only on $n$, $p$ and $\Omega$ such that 
\begin{align}
	&\label{eq:wtauL2}
	\int_{\tau_*}^\infty \int_{\Omega(\tau)} 
	|\partial_{\tau} w_{(\tilde x,T)}|^2 \rho d\eta d\tau 
	\leq C(C_*+C_*'), \\
	&\label{eq:wL2Linfty}
	\sup_{\tau\geq \tau_*+2} 
	\int_{\Omega(\tau)} |w_{(\tilde x,T)}(\eta,\tau)|^2 \rho(\eta) d\eta 
	\leq C( C_*+C_*' )^\frac{1}{p}
	+ C (C_*+C_*'), 
\end{align} 
for $\tilde x\in \overline{\Omega}$, 
where  $\tau_*$, $C_*$ and $C_*'$ are given in \eqref{eq:stars} 
and are independent of $\tilde x$. 
\end{lemma}

\begin{proof}
We first prove \eqref{eq:wtauL2} 
in  the interior case $\tilde x\in\Omega$. 
Let $T/2<\tilde t<T$. 
From \eqref{eq:GKoriginal} and  
performing the scale back to the original variables 
in the boundary integral, it follows that 
\[
\begin{aligned}
	\cE_{(\tilde x,\tilde t)}(\tau_2) 
	&= \cE_{(\tilde x,\tilde t)}(\tau_1) 
	- \int_{\tau_1}^{\tau_2} \int_{\Omega(\tau)} 
	|\partial_\tau w_{(\tilde x,\tilde t)}|^2 \rho d\eta d\tau 
	-\frac{1}{4} \int_{t_1}^{t_2} B_{(\tilde x,\tilde t)} (t) dt 
\end{aligned}
\]
for $\tau_*\leq \tau_1<\tau_2<\infty$ 
and $\tilde t-\delta^2\leq t_1<t_2<\tilde t$ with 
$\tau_1=-\log(\tilde t-t_1)$ and $\tau_2=-\log(\tilde t-t_2)$, 
where $B_{(\tilde x,\tilde t)}$ 
is defined by \eqref{eq:Bdef}. 
By \eqref{eq:EcErel}, 
$B_{(\tilde x,T)}=B_{(\tilde x,T)}^+ - B_{(\tilde x,T)}^-$ 
(see \eqref{eq:Bpmdef} for the definition 
of $B_{(\tilde x,T)}^\pm$) and Lemma \ref{lem:GKul}, 
\[
\begin{aligned}
	\int_{\tau_1}^{\tau_2} \int_{\Omega(\tau)} 
	|\partial_\tau w_{(\tilde x,\tilde t)}|^2 \rho d\eta d\tau 
	&= \cE_{(\tilde x,\tilde t)}(\tau_1) - \cE_{(\tilde x,\tilde t)}(\tau_2) 
	-\frac{1}{4} \int_{t_1}^{t_2} B_{(\tilde x,\tilde t)} (t) dt \\
	&\leq 
	2C_*+C_*'+ \frac{1}{4} \sup_{\tilde x\in\Omega}
	\int_{\tilde t-\delta^2}^{\tilde t} B_{(\tilde x,\tilde t)}^- (t) dt. 
\end{aligned}
\]
From Lemma \ref{lem:B-bybeta}, \eqref{eq:betaBsupE} and  \eqref{eq:tiltbddEpr}, 
it follows that 
\begin{equation}\label{eq:Bminussup}
\begin{aligned}
	\sup_{\tilde x\in \Omega}\int_{\tilde t-\delta^2}^{\tilde t} 
	B_{(\tilde x,\tilde t)}^- (t) dt 
	&\leq 
	C \delta \sup_{0<r<\delta/\sqrt{2}} \beta(r) \\
	&\leq 
	C \delta \sup_{\tilde x\in \Omega} E_{(\tilde x,\tilde t)}(\tilde t-\delta^2) 
	\leq 
	C C_*, 
\end{aligned}
\end{equation}
where $C>0$ is a constant depending only on $n$, $p$ and $\Omega$. 
Therefore,  
\begin{equation}\label{eq:tau1tau2wtau}
\begin{aligned}
	\int_{\tau_1}^{\tau_2} \int_{\Omega(\tau)} 
	|\partial_\tau w_{(\tilde x,\tilde t)}|^2 \rho d\eta d\tau 
	\leq 
	2C_*+C_*' + C C_* 
\end{aligned}
\end{equation}
for $T/2<\tilde t<T$, 
$\tau_*\leq \tau_1<\tau_2<\infty$ and 
$\tilde t-\delta^2\leq t_1<t_2<\tilde t$ with 
$\tau_1=-\log(\tilde t-t_1)$ and $\tau_2=-\log(\tilde t-t_2)$. 
By the same reasoning as in the final part of the proof of Proposition \ref{pro:qmono}, 
we can take the limit $\tilde t \to T$, 
and so \eqref{eq:tau1tau2wtau} also holds for $\tilde t=T$. 
Therefore, 
in the interior case $\tilde x\in\Omega$, 
we obtain 
\[
\begin{aligned}
	\int_{\tau_*}^\infty \int_{\Omega(\tau)} 
	|\partial_\tau w_{(\tilde x,T)}|^2 \rho d\eta d\tau 
	\leq 
	2C_*+C_*' + C C_*. 
\end{aligned}
\]
Since the right-hand side is independent of the choice of $\tilde x\in \Omega$, 
by taking the limit as $\tilde x$ approaches $\partial \Omega$ 
with the help of \eqref{eq:regu}, 
the same bound also holds for the boundary case $\tilde x\in\partial \Omega$. 
Then we obtain \eqref{eq:wtauL2}.

We next show \eqref{eq:wL2Linfty} by using $W$ in \eqref{eq:Wksdef}. 
Let $\tilde x\in \overline{\Omega}$ and $\tau\geq \tau_*+2$. 
By the $1$-dimensional Sobolev inequality, there exists $C>0$ depending only on 
the length of the time interval 
$|(\tau-2,\tau)|=2$ such that 
\begin{equation}\label{eq:1dSobolev}
	W^2(\tau) 
	\leq 
	\left( \sup_{\sigma\in [\tau-2,\tau]} W(\sigma) \right)^2
	\leq C(\| W \|_{{L^2 (\tau-2,\tau)}}^2 + \| W' \|_{{L^2 (\tau-2,\tau)}}^2)
\end{equation}
for $\tau\ge \tau_* +2$. 
We estimate $\| W \|_{{L^2 (\tau-2,\tau)}}^2$ and $\| W' \|_{{L^2 (\tau-2,\tau)}}^2$. 
As for $\| W \|_{{L^2 (\tau-2,\tau)}}^2$, 
we see from \eqref{eq:w2GK} and Lemma \ref{lem:GKul} that 
\begin{equation}\label{eq:p1main}
\begin{aligned}
	\int_{\Omega(\tau)} |w|^{p+1} \rho d\eta
	&= 
	\frac{p+1}{p-1} \int_{\Omega(\tau)} w w_\tau \rho d\eta 
	+\frac{2(p+1)}{p-1} \cE_{(\tilde x,T)}(\tau) \\
	&\leq 
	C \int_{\Omega(\tau)} |w w_\tau| \rho d\eta 
	+ C(C_*+C_*') 
\end{aligned}
\end{equation}
for $\tau\geq \tau_*$, 
where $C>0$ depends only on $p$. 
From integrating this inequality and using the Young inequality, 
it follows that for $\tau\geq \tau_*+2$, 
\[
\begin{aligned}
	&\int_{\tau-2}^\tau 
	\int_{\Omega(\sigma)} |w|^{p+1} \rho d\eta d\sigma
	\leq 
	C \iint |w w_\tau| \rho d\eta d\sigma 
	+ C(C_*+C_*')  \\
	&\leq 
	\frac{1}{2} \iint |w|^{p+1} \rho d\eta d\sigma
	+ C \iint  |w_\tau|^\frac{p+1}{p} 
	\rho d\eta d\sigma
	+ C(C_*+C_*'), 
\end{aligned}
\]
where $\iint \ldots d\eta d\sigma 
= \int_{\tau-2}^\tau \int_{\Omega(\sigma)} \ldots d\eta d\sigma$ 
and $C>0$ varies from line to line and 
depends only on $n$, $p$ and $\Omega$. 
Then by the H\"older inequality, 
$\int_{\R^n} \rho d\eta \leq C_n$ ($C_n$ depends only on $n$) 
and \eqref{eq:wtauL2}, 
\begin{equation}\label{eq:wp1sptbdd}
\begin{aligned}
	\int_{\tau-2}^\tau 
	\int_{\Omega(\sigma)} |w|^{p+1} \rho d\eta d\sigma
	&\leq 
	C\left( \iint  |w_\tau|^2 
	\rho d\eta d\sigma \right)^\frac{p+1}{2p}
	+ C(C_*+C_*') \\
	&\leq 
	C( C_*+C_*' )^\frac{p+1}{2p}
	+ C(C_*+C_*'). 
\end{aligned}
\end{equation}
This together with the H\"older inequality implies 
the existence of $C>0$ depending only on $n$, $p$ and $\Omega$ 
such that  for $\tau\geq \tau_*+2$, 
\begin{equation}\label{eq:WL2sptime1}
\begin{aligned}
	\|W\|_{L^2(\tau-2,\tau)}^2 
	&= 
	\int_{\tau-2}^{\tau} \int_{\Omega(\sigma)} 
	w^2 \rho d\eta d\sigma  \\
	&\leq 
	C( C_*+C_*' )^\frac{1}{p}
	+ C (C_*+C_*')^\frac{2}{p+1}. 
\end{aligned}
\end{equation}

As for $\| W' \|_{{L^2 (\tau-2,\tau)}}^2$, 
since $w(\cdot,\tau)=0$ on $\partial \Omega(\tau)$, 
direct computations and the H\"older inequality  show that 
\[
	|W'(\tau)| = 
	\frac{1}{W(\tau)}\left|  \int_{\Omega(\tau)} w w_\tau \rho d\eta  \right|
	\leq 
	\left( \int_{\Omega(\tau)} |w_\tau|^2  \rho d\eta \right)^\frac{1}{2}, 
\]
and so by \eqref{eq:wtauL2}, 
\[
	\|W'\|_{L^2(\tau-2,\tau)}^2 
	\leq 
	\int_{\tau-2}^\tau \int_{\Omega(\sigma)} |w_\tau|^2  \rho d\eta d\sigma
	\leq C (C_*+C_*'). 
\]
This together with \eqref{eq:1dSobolev} and \eqref{eq:WL2sptime1}
shows that 
\[
	W^2(\tau) 
	\leq 
	C( C_*+C_*' )^\frac{1}{p}
	+ C (C_*+C_*')^\frac{2}{p+1} 
	+ C (C_*+C_*') 
\]
for $\tau\geq \tau_*+2$. 
Hence \eqref{eq:wL2Linfty} holds, since $1/p<2/(p+1)<1$. 
\end{proof}

Let us next show the following estimate 
which corresponds to the base step of the inductive estimates 
in the next section 
(see Proposition \ref{pro:boot} below).

\begin{proposition}\label{pro:base}
There exists $C>0$ depending only on $n$, $p$ and $\Omega$ such that  
for $\tilde x\in \overline{\Omega}$, 
\[
	\sup_{\tau \geq \tau_*+2}
	\int _{\tau-2}^{\tau} 
	\left(\int_{\Omega(\sigma)} |w_{(\tilde x,T)}|^{p+1} \rho
	 d\eta \right)^2 d\sigma 
	\leq 
	C( C_*+C_*' )^\frac{p+1}{p}
	+ C (C_*+C_*')^2, 
\]
where  $\tau_*$, $C_*$ and $C_*'$ are given in \eqref{eq:stars} 
and are independent of $\tilde x\in \overline{\Omega}$. 
\end{proposition}

\begin{proof}
From \eqref{eq:p1main} and \eqref{eq:wL2Linfty}, it follows that 
\[
\begin{aligned}
	&\int_{\Omega(\tau)} |w|^{p+1} \rho d\eta
	\leq 
	C \int_{\Omega(\tau)} |w w_\tau| \rho d\eta 
	+ C(C_*+C_*') \\
	&\leq 
	C\left( C'( C_*+C_*' )^\frac{1}{p}
	+ C' (C_*+C_*') \right)^\frac{1}{2}
	\left( \int_{\Omega (\tau)} |w_\tau|^2 \rho d\eta \right)^\frac{1}{2}
	+ C(C_*+C_*') 
\end{aligned}
\]
for $\tau\geq \tau_*+2$. Then by \eqref{eq:wtauL2}, we have 
\[
\begin{aligned}
	&\int_{\tau-2}^\tau 
	\left( \int_{\Omega(\sigma)} |w|^{p+1} \rho d\eta \right)^2 d\sigma \\
	&\leq 
	C\left( ( C_*+C_*' )^\frac{1}{p}
	+ (C_*+C_*') \right)
	\int_{\tau-2}^\tau \int_{\Omega(\sigma)} |w_\tau|^2 \rho d\eta d\sigma
	+ C(C_*+C_*')^2 \\
	&\leq 
	C\left( ( C_*+C_*' )^\frac{1}{p}
	+ (C_*+C_*') \right)(C_*+C_*')
	+ C(C_*+C_*')^2. 
\end{aligned}
\]
This gives the desired estimate. 
\end{proof}

Next, for proving the inductive estimates in the next section 
(particularly, for the proof of Lemma \ref{lem:maximal} below), 
we need to localize the spatial integrals. 
Therefore, we introduce a localization of the Giga--Kohn energy 
and give its upper bound. 
Let $0<R<1$ be a constant 
which will be specified in \eqref{eq:Rcond} below. 
For $r>0$,  
we write $B_r:=B_r(0)=\{ y\in\R^n;  |y|<r\}$.
For $k\geq1$, 
let $\psi_k\in C_0^\infty(\R^n)$ satisfy 
$0\leq \psi_k\leq 1$ in $\R^n$, $\psi_k=1$ in $B_{R/2^k}$, 
$\psi_k=0$ in $\R^n\setminus \overline{B_{R/2^{k-1}}}$ and 
\begin{equation}\label{eq:nabpsibd}
	|\nabla \psi_k(\eta)| \leq  \frac{2^k C}{R}, 
	\quad 
	|\nabla^2 \psi_k(\eta)| \leq  \frac{4^k C}{R^2} 
	\quad \mbox{ for }\eta\in \R^n, 
\end{equation}
with an absolute constant $C>0$. 
We define the localization of $\cE_{(\tilde x, T)}$ by 
\begin{equation}\label{eq:cElocdef}
	\cE_{(\tilde x, T)}^{\psi_k} (\tau) := 
	\int_{\Omega(\tau)} \left( 
	\frac{ |\nabla w|^2}{2}
	-\frac{|w|^{p+1}}{p+1} 
	+\frac{w^2}{2(p-1)} \right) 
	\rho \psi_k^4 d\eta 
\end{equation}
for $\tilde x\in \overline{\Omega}$, 
where the original Giga--Kohn energy $\cE_{(\tilde x, T)}$ 
has been defined in \eqref{eq:cEdef} with $\tilde t=T$ 
and the exponent $4$ of $\psi_k^4$ will be used 
for computing $\Delta(\rho^{p/(p+1)} \psi_k^{4p/(p+1)})$ 
in \eqref{eq:Fcomputation} below. 
Let us estimate $\cE_{(\tilde x, T)}^{\psi_k}$.

\begin{lemma}\label{lem:locEunibdd}
Let $R$ be the constant in \eqref{eq:nabpsibd}. 
Then there exists $C>0$ depending only on $n$, $p$ and $\Omega$ 
such that for $k \geq1$, 
\[
	\sup_{\tau\geq \tau_* + 2} \cE^{\psi_k}_{(\tilde x,T)}(\tau)
	\leq 
	\frac{4^k C}{R^2}  ( C_*+C_*' )^\frac{p+1}{2p} + \frac{4^k C}{R^2}  (C_*+C_*'), 
\] 
where  $\tau_*$, $C_*$ and $C_*'$ are given in \eqref{eq:stars} 
and are independent of $\tilde x\in \overline{\Omega}$. 
\end{lemma}

\begin{proof}
Let $\tilde x\in \overline{\Omega}$ and $\tau\geq \tau_*+2$. By the mean value theorem, 
there exists $\tau-1\leq \tau'\leq \tau$ such that 
$\cE^{\psi_k}(\tau')=\int_{\tau-1}^\tau \cE^{\psi_k}(\sigma) d\sigma$. 
We note that the idea of using the mean value theorem is 
due to Giga--Matsui--Sasayama \cite[Lemma 5.6]{GMS04}. 
Then, 
\begin{equation}\label{eq:cEmeanesti}
\begin{aligned}
	\cE^{\psi_k}(\tau)
	&=\cE^{\psi_k}(\tau') 
	+ \int_{\tau'}^\tau \frac{d \cE^{\psi_k}}{d\tau}(\sigma) d\sigma  \\
	&= \int_{\tau-1}^\tau \cE^{\psi_k}(\sigma) d\sigma 
	+ \int_{\tau'}^\tau \frac{d \cE^{\psi_k}}{d\tau}(\sigma) d\sigma. 
\end{aligned}
\end{equation}
By \eqref{eq:cElocdef}, $\psi_k\leq 1$, \eqref{eq:cEdef}, 
Lemma \ref{lem:GKul} and \eqref{eq:wp1sptbdd},  
there exists $C>0$ depending only on $n$, $p$ and $\Omega$ such that 
\begin{equation}\label{eq:cEt1tbddstar}
\begin{aligned}
	\int_{\tau-1}^\tau \cE^{\psi_k}(\sigma) d\sigma  
	&\leq 
	\int_{\tau-1}^\tau \int_{\Omega(\sigma)} \left( 
	\frac{ |\nabla w|^2}{2}
	+\frac{w^2}{2(p-1)} \right) 
	\rho d\eta d\sigma \\
	&= \int_{\tau-1}^\tau \cE(\sigma) d\sigma
	+ \int_{\tau-1}^\tau\int_{\Omega(\sigma)} \frac{|w|^{p+1}}{p+1} \rho d\eta d\sigma \\
	&\leq 
	C(C_*+C_*') + C( C_*+C_*' )^\frac{p+1}{2p}. 
\end{aligned}
\end{equation}

As for the estimate of $\int_{\tau'}^\tau (d \cE^{\psi_k} / d\tau) (\sigma) d\sigma$, 
we compute the derivative of $\cE^{\psi_k}$ 
under the abbreviation $\int\ldots = \int_{\Omega(\tau)} \ldots d\eta$. 
The differentiations below are justified 
by the standard parabolic regularity with \eqref{eq:regu}.
Integrating by parts and using $w=0$ on $\partial\Omega(\tau)$, 
we see that 
\[
\begin{aligned}
	\frac{d}{d\tau} \int
	\frac{ |\nabla w|^2}{2}
	\rho \psi_k^4 
	&= 
	- \frac{d}{d\tau} \int 
	\frac{w}{2}
	\nabla \cdot ( \rho \nabla w ) \psi_k^4 
	- 2 \frac{d}{d\tau} \int w \rho \psi_k^3 \nabla w \cdot \nabla \psi_k   \\
	&= 
	- \int \frac{w_\tau}{2}
	\nabla \cdot ( \rho \nabla w ) \psi_k^4 
	- \int  \frac{w}{2}
	\nabla \cdot ( \rho \nabla w_\tau ) \psi_k^4  \\
	&\quad 
	- 2\int  w_\tau \rho \psi_k^3 \nabla w \cdot \nabla \psi_k 
	- 2 \int w \rho \psi_k^3 \nabla w_\tau \cdot \nabla \psi_k. 
\end{aligned}
\]
Again by the integration by parts for the terms including $\nabla w_\tau$, 
\[
\begin{aligned}
	&\frac{d}{d\tau} \int 
	\frac{ |\nabla w|^2}{2}
	\rho \psi_k^4   \\
	&= 
	- \int
	\frac{w_\tau}{2}
	\nabla \cdot ( \rho \nabla w ) \psi_k^4  
	+ \int 
	\frac{\nabla w}{2}
	\cdot ( \rho \nabla w_\tau ) \psi_k^4  
	+ 2 \int w \rho \psi_k^3 \nabla w_\tau \cdot \nabla \psi_k \\
	&\quad 
	- 2\int w_\tau \rho \psi_k^3 \nabla w \cdot \nabla \psi_k 
	+ 2 \int w_\tau \rho \psi_k^3 \nabla w \cdot \nabla \psi_k 
	+ 2\int w  w_\tau \nabla \cdot (\rho \psi_k^3 \nabla \psi_k)  \\
	&= - \int  w_\tau  \nabla \cdot ( \rho \nabla w ) \psi_k^4 
	- 4\int  w_\tau \rho \psi_k^3 \nabla w \cdot \nabla \psi_k  \\
	&\quad + \frac{1}{2} \int_{\partial \Omega(\tau)} 
	w_\tau \rho \psi_k^4  \nabla w\cdot \nu dS(\eta), 
\end{aligned}
\]
where $\nu$ is the outward unit normal at $\eta\in \partial \Omega(\tau)$. 
Then by \eqref{eq:cElocdef} and \eqref{eq:wrhoeq}, 
we can check that 
\[
	\frac{d\cE^{\psi_k}}{d\tau}  
	= 
	- \int  |w_\tau|^2 \rho \psi_k^4 
	- 4\int  w_\tau \rho \psi_k^3 \nabla w \cdot \nabla \psi_k 
	+ \frac{1}{2} \int_{\partial \Omega(\tau)} 
	w_\tau \rho \psi_k^4  \nabla w\cdot \nu dS(\eta). 
\]
Since $u_t=0$ on $\partial\Omega$ and 
\[
	u(x,t)=(T-t)^{-\frac{1}{p-1}} 
	w( (T-t)^{-\frac{1}{2}}(x-\tilde x), -\log(T-t))
\] 
by \eqref{eq:backw}, 
we have $w_\tau = -(\nabla w \cdot \eta)/2$ on $\partial\Omega(\tau)$. 
Note that the tangential derivatives of $w$ vanish, and so 
$\nabla w=(\nabla w\cdot \nu)\nu$, 
$\nabla w \cdot \eta= (\nabla w\cdot \nu) (\eta\cdot \nu)$ 
and $|\nabla w|^2=|\nabla w \cdot \nu|^2$ hold 
on $\partial\Omega(\tau)$. 
Hence we obtain 
\begin{equation}\label{eq:locmonoequality}
\begin{aligned}
	\frac{d\cE^{\psi_k}}{d\tau} 
	&= 
	- \int_{\Omega(\tau)}  |w_\tau|^2 \rho \psi_k^4 d\eta
	- 4\int_{\Omega(\tau)}  w_\tau \rho \psi_k^3 \nabla w \cdot \nabla \psi_k d\eta \\
	&\quad 
	-\frac{1}{4} \int_{\partial \Omega(\tau)} 
	(\eta\cdot \nu) |\nabla w|^2 \rho \psi_k^4   dS(\eta). 
\end{aligned}
\end{equation}

By the Young inequality and by performing the scale back to the original variables 
for the boundary integral (see also \eqref{eq:GKoriginal}, 
\eqref{eq:monoBdiff} and \eqref{eq:Bdef}), 
\[
\begin{aligned}
	\int_{\tau'}^\tau \frac{d \cE^{\psi_k}}{d\tau}(\sigma) d\sigma
	&\leq  
	4 \int_{\tau'}^\tau \int_{\Omega(\sigma)} 
	|\nabla w|^2 \rho \psi_k^2 |\nabla \psi_k|^2 d\eta d\sigma \\
	&\quad 
	+\frac{1}{4} \int_{\tau'}^\tau \int_{\partial \Omega(\sigma)} 
	(\eta\cdot \nu)_{-} |\nabla w|^2 \rho   dS(\eta) d\sigma \\
	&= 
	4 \int_{\tau'}^\tau \int_{\Omega(\sigma)} 
	|\nabla w|^2 \rho \psi_k^2 |\nabla \psi_k|^2 d\eta d\sigma 
	+ \frac{1}{4} \int_{t'}^{t} B_{(\tilde x,T)}^- (t) dt, 
\end{aligned}
\]
where $\tau=-\log(T-t)$ and $\tau'=-\log(T-t')$. 
From \eqref{eq:nabpsibd}, \eqref{eq:cEdef} with 
the estimate in \eqref{eq:cEt1tbddstar} and \eqref{eq:Bminussup}, 
it follows that 
\[
\begin{aligned}
	\int_{\tau'}^\tau \frac{d \cE^{\psi_k}}{d\tau}(\sigma) d\sigma
	&\leq  
	\frac{4^k C}{R^2} \int_{\tau-1}^\tau \int_{\Omega(\sigma)} 
	\frac{|\nabla w|^2}{2} \rho d\eta d\sigma 
	+ C C_* \\
	&\leq 
	\frac{4^k C}{R^2} ( C_*+C_*' ) + \frac{4^k C}{R^2} (C_*+C_*')^\frac{p+1}{2p}
	+ C C_*, 
\end{aligned}
\]
where $C>0$ depends only on $n$, $p$ and $\Omega$. 
Then by this estimate, \eqref{eq:cEmeanesti}, \eqref{eq:cEt1tbddstar}, 
$0<R<1$ and replacing $C$ with a larger constant 
depending only on $n$, $p$ and $\Omega$, 
we obtain the desired estimate. 
\end{proof}

Finally in this section, we give 
the following localized integral estimate for $|w|^{p+1}$. 
This is a key ingredient for showing the inductive estimates 
in the next section.

\begin{lemma}\label{lem:w2deri}
Let $R$ be the constant in \eqref{eq:nabpsibd}. 
Then there exists $C>0$ depending only on $n$, $p$ and $\Omega$ 
such that 
\[
	\int_{\Omega(\tau)} |w|^{p+1} \rho \psi_k^4 d\eta
	\leq 
	C  \int_{\Omega(\tau)} |w w_\tau| \rho \psi_k^4 d\eta 
	+ \frac{4^k C}{R^2} (C_*+C_*') + \frac{4^k C}{R^2} ( C_*+C_*' )^\frac{1}{p}
\]
for $\tau\geq \tau_*+2$ and $k\geq1$, 
where  $\tau_*$, $C_*$ and $C_*'$ are given in \eqref{eq:stars} 
and are independent of $\tilde x\in \overline{\Omega}$. 
\end{lemma}

\begin{proof}
Let $\tilde x\in \overline{\Omega}$ and $\tau\geq \tau_*+2$. 
From $w=0$ on $\partial \Omega(\tau)$, \eqref{eq:wrhoeq}, 
the integration by parts and 
the abbreviation $\int \ldots= \int_{\Omega(\tau)}\ldots d\eta$, 
it follows that 
\[
\begin{aligned}
	&\frac{1}{2} \frac{d}{d\tau} 
	\int_{\Omega(\tau)} |w|^2 \rho \psi_k^4 d\eta 
	= 
	\int w w_\tau \rho \psi_k^4  \\
	&= 
	\int w \psi_k^4 \nabla \cdot (\rho \nabla w) 
	- \frac{1}{p-1} \int |w|^2 \psi_k^4 \rho 
	+ \int |w|^{p+1} \rho \psi_k^4  \\
	&= 
	- 2 \int \left( \frac{|\nabla w|^2}{2} 
	+ \frac{|w|^2}{2(p-1)} \right)
	\rho \psi_k^4 
	+ \int  |w|^{p+1} \rho \psi_k^4 
	- 4 \int  w \rho \psi_k^3 \nabla w\cdot \nabla \psi_k. 
\end{aligned}
\]
Then by \eqref{eq:cElocdef}, we have 
\begin{equation}\label{eq:locwp1eq}
\begin{aligned}
	\frac{1}{2} \frac{d}{d\tau} 
	\int_{\Omega(\tau)} |w|^2 \rho \psi_k^4 d\eta 
	&= 
	- 2 \cE_{(\tilde x,T)}^{\psi_k}(\tau) 
	+ \frac{p-1}{p+1} \int |w|^{p+1} \rho \psi_k^4  \\
	&\quad 
	- 4 \int  w \rho \psi_k^3 \nabla w\cdot \nabla \psi_k. 
\end{aligned}
\end{equation}

By \eqref{eq:locwp1eq}, 
Lemma \ref{lem:locEunibdd} and 
the Young inequality, we see that 
\[
\begin{aligned}
	\frac{p-1}{p+1} \int |w|^{p+1} \rho \psi_k^4 
	&= 
	\int w w_\tau \rho \psi_k^4 
	+ 2 \cE_{(\tilde x,T)}^{\psi_k}(\tau) 
	+ 4 \int w \rho \psi_k^3 \nabla w\cdot \nabla \psi_k \\
	&\leq 
	\int w w_\tau \rho \psi_k^4  
	+ 
	\frac{4^k C}{R^2} ( C_*+C_*' )^\frac{p+1}{2p} + \frac{4^k C}{R^2}(C_*+C_*') \\
	&\quad 
	+ \frac{p-1}{2} \int  \frac{|\nabla w|^2}{2} \rho \psi_k^4 
	+ C \int |w|^2 \rho \psi_k^2|\nabla \psi_k|^2, 
\end{aligned}
\]
where $C>0$ depends only on $n$, $p$ and $\Omega$. 
For the fourth and fifth terms in the right-hand side, 
by \eqref{eq:cElocdef}, \eqref{eq:nabpsibd}, 
Lemma \ref{lem:locEunibdd} and \eqref{eq:wL2Linfty},  
\[
\begin{aligned}
	&\frac{p-1}{2} \int \frac{|\nabla w|^2}{2} \rho \psi_k^4 
	+ C \int |w|^2 \rho \psi_k^2 |\nabla \psi_k|^2  \\
	&\leq 
	\frac{p-1}{2}\left( \cE_{(\tilde x,T)}^{\psi_k}(\tau) 
	+ \int \frac{|w|^{p+1}}{p+1} \rho \psi_k^4 \right)
	+ \frac{4^k C}{R^2} \int |w|^2 \rho \\
	&\leq 
	\frac{p-1}{2(p+1)} \int |w|^{p+1} \rho \psi_k^4  
	+ \frac{4^k C}{R^2} ( C_*+C_*' )^\frac{p+1}{2p}  \\
	&\quad 
	+ \frac{4^k C}{R^2} (C_*+C_*') + \frac{4^k C}{R^2} ( C_*+C_*' )^\frac{1}{p}. 
\end{aligned}
\]
Hence the desired inequality follows, 
since $1/p<(p+1)/(2p)<1$. 
\end{proof}

\section{Nonexistence of type II blow-up}\label{sec:rate}
In this section, we show an inductive estimate, where 
the basis of the inductive steps is Proposition \ref{pro:base}, 
and then we prove Theorem \ref{th:main}. 
We remark that 
our inductive steps will be done by finitely many steps. 
In addition, all the estimates, except for the ones in the 
proof of Theorem \ref{th:main}, are valid for any $p>1$.

Throughout this section, we consider the interior case  $\tilde x\in \Omega$, 
fix $\delta$ in Proposition \ref{pro:qmono} and 
set $\tau_*$, $C_*$ and $C_*'$ as in \eqref{eq:stars}. 
Let $0<R<1$ be a constant 
which will be specified in \eqref{eq:Rcond} below. 
For $k\geq1$,  let $\psi_k\in C_0^\infty(\R^n)$ satisfy 
$0\leq \psi_k\leq 1$ in $\R^n$, $\psi_k=1$ in $B_{R/2^k}$, 
$\psi_k=0$ in $\R^n\setminus \overline{B_{R/2^{k-1}}}$ and 
\eqref{eq:nabpsibd}. 
Note that $\psi_k \leq \psi_{k-1}$ on $\R^n$ and 
$\psi_{k-1}=1$ on $\supp \psi_k$. 
Our first goal is to prove the following inductive steps:

\begin{proposition}\label{pro:boot}
Let $\tilde x\in \Omega$. 
Suppose that there exist $q\geq 2$, $k\geq1$ and $C_k'>0$ 
independent of $\tilde x\in \Omega$ and 
satisfying 
\begin{equation}\label{eq:assumpq2}
	\sup_{\tau\geq \tau_*+2}\int_{\tau-1-2^{-(k-1)}}^{\tau} 
	\left( \int_{\Omega(\sigma)} |w_{(\tilde x,T)}|^{p+1} 
	\rho \psi_{k-1}^4 d\eta \right)^q
	d\sigma 
	\leq C_k'. 
\end{equation}
Then there exists $C_k''>0$ depending only on 
$k$, $n$, $p$, $q$, $\Omega$, $R$, $\tau_*$, $C_*$, $C_*'$ and $C_k'$ 
(in particular, independent of $\tilde x\in \Omega$) such that 
\begin{equation}\label{eq:assumpqp1}
	\sup_{\tau\geq \tau_*+2}\int_{\tau-1-2^{-k}}^{\tau} 
	\left( \int_{\Omega(\sigma)} |w_{(\tilde x,T)}|^{p+1} \rho \psi_k^4 d\eta 
	\right)^{q+\frac{1}{p+1}} d\sigma \leq C_k''. 
\end{equation}
\end{proposition}

In what follows, 
the constants $C_k>0$ may change even within the same line 
and depend only on $k$, $n$, $p$, $q$, $\Omega$, $R$, $\tau_*$, $C_*$, $C_*'$ 
and $C_k'$. 
Moreover, the constants 
$C>0$ depend only on $n$, $p$, $q$, $\Omega$, $R$, $\tau_*$, $C_*$ and $C_*'$.

To prove \eqref{eq:assumpqp1} under \eqref{eq:assumpq2}, we estimate 
$\int_{\Omega(\tau)} |w|^{p+1} \rho \psi_k^4 d\eta$. 
By Lemma \ref{lem:w2deri}, for each $\tau\geq \tau_*+2$, 
\begin{equation}\label{eq:p122tau}
\begin{aligned}
	\int_{\Omega(\tau)} |w|^{p+1} \rho \psi_k^4 d\eta
	&\leq 
	C  \int_{\Omega(\tau)} |w w_\tau| \rho \psi_k^4 d\eta 
	+ C_k \\
	&= 
	C \| [w w_\tau \rho \psi_k^4](\cdot,\tau) \|_{L_\eta^1}
	+ C_k, 
\end{aligned}
\end{equation} 
where $L_\eta^1:=L^1(\Omega(\tau))$. 
For $p>1$ and $q\geq 2$, let 
\begin{equation}\label{eq:lamchoice}
	\lambda := p+1 - \frac{p-1}{q+1}-\eps, 
\end{equation}
where we take $\eps>0$ such that  
$2<\lambda<p+1$ holds. 
Note that $(p+1)/p < \lambda/(\lambda-1) <2$ holds and that 
$\eps$ will be replaced with a smaller constant depending only on 
$p$ and $q$.
Then, the H\"older inequality and the interpolation inequality 
with respect to the measure $d\rho_k:= \rho\psi_k^4 d\eta$ 
show that 
\begin{equation}\label{eq:inter}
\begin{aligned}
	&\| [w w_\tau \rho  \psi_k^4](\cdot,\tau) \|_{L_\eta^1}
	\leq 
	\| w \rho^\frac{1}{\lambda} \psi_k^\frac{4}{\lambda}\|_{L_\eta^\lambda}
	\| w_\tau \rho^\frac{\lambda-1}{\lambda}\psi_k^\frac{4(\lambda-1)}{\lambda}
	\|_{L_\eta^\frac{\lambda}{\lambda-1}} \\
	&= 
	\| w \rho^\frac{1}{\lambda} \psi_k^\frac{4}{\lambda}\|_{L_\eta^\lambda}
	\| w_\tau \|_{L_{\rho_k}^\frac{\lambda}{\lambda-1} }
	\leq 
	\| w \rho^\frac{1}{\lambda} \psi_k^\frac{4}{\lambda}\|_{L_\eta^\lambda}
	\| w_\tau \|_{L_{\rho_k}^\frac{p+1}{p}}^\theta 
	\| w_\tau \|_{L_{\rho_k}^2}^{1-\theta}, 
\end{aligned}
\end{equation}
where $L_{\rho_k}^{\lambda/(\lambda-1)}$ 
is the $L^{\lambda/(\lambda-1)}(\Omega(\tau))$ space 
with respect to the measure $d\rho_k= \rho\psi_k^4 d\eta$ and 
\begin{equation}\label{eq:thetadef}
	\theta:= \frac{(p+1)(\lambda-2)}{(p-1)\lambda}, 
	\quad 
	1-\theta = 
	\frac{2(p+1-\lambda)}{(p-1)\lambda}. 
\end{equation}
We show the following estimate 
by a Cazenave--Lions \cite{CL84} type interpolation.

\begin{lemma}\label{lem:CLinter}
Suppose that \eqref{eq:assumpq2} holds.  
Then, 
\[
	\sup_{\tau\in[\tilde \tau-1-2^{-k},\tilde \tau]} 
	\| [w \rho^\frac{1}{\lambda} \psi_k^\frac{4}{\lambda}](\cdot,\tau)
	\|_{L_\eta^\lambda} 
	\leq C_k
	\quad \mbox{ for }\tilde \tau\geq \tau_*+2, 
\]
where $C_k>0$ is independent of $\tilde \tau$ 
and $L_\eta^\lambda:=L^\lambda(\Omega(\tau))$.
\end{lemma}

\begin{proof}
Let $\tilde \tau\geq \tau_*+2$. 
We prepare a cut-off function 
$\zeta_k(\tau):=\tilde \zeta_k(\tau-\tilde \tau)$. 
Here we take $\tilde \zeta_k\in C^\infty(\R)$ such that 
$\tilde \zeta_k(\sigma)=0$ for $\sigma \leq -1-2^{-(k-1)}$, 
$0\leq \tilde \zeta_k(\sigma) \leq 1$ for $-1-2^{-(k-1)}< \sigma < -1-2^{-k}$ and 
$\tilde \zeta_k(\sigma) = 1$ for $\sigma \geq -1-2^{-k}$. 
Note that 
$\supp \zeta_k\subset [\tilde \tau-1-2^{-(k-1)},\infty)$ 
and that bounds of the derivatives of $\zeta_k$ depend only on $k$ 
and are independent of $\tilde \tau$. 
Taking $w=0$ on $\partial \Omega(\tau)$ into account, 
we extend $w(\cdot,\tau)$ by $0$ to the whole space $\R^n$ 
in the proof of this lemma. 
Then, 
\[
\begin{aligned}
	\sup_{\tau\in[\tilde \tau-1-2^{-k},\tilde \tau]} 
	\| [w \rho^\frac{1}{\lambda} \psi_k^\frac{4}{\lambda}](\cdot,\tau)\|_{L_\eta^\lambda} 
	\leq 
	\sup_{\tau\in[\tilde \tau-2,\tilde \tau]} 
	\| [w \rho^\frac{1}{\lambda} \psi_k^\frac{4}{\lambda} 
	\zeta_k](\cdot,\tau)
	\|_{L^\lambda(\R^n)}. 
\end{aligned}
\]

Since $2<\lambda<p+1-(p-1)/(q+1)$ by \eqref{eq:lamchoice}, 
we can apply a Cazenave--Lions \cite{CL84} type interpolation. 
See \cite[(2.9)]{Qu03} and also \cite[(51.6)]{QSbook2} 
for the embedding used here, which follows from 
the general results in \cite[Theorems VII.2.7.2, VII.2.8.3]{Ambook}. 
Then, there exists a constant $C'>0$ depending only on 
$n$, $p$, $q$, $\lambda$ and the length of the time interval 
$|[\tilde \tau-2, \tilde \tau]|=2$ such that 
\[
\begin{aligned}
	&\sup_{\tau\in[\tilde \tau-2,\tilde \tau]} 
	\| [w \rho^\frac{1}{\lambda} \psi_k^\frac{4}{\lambda} 
	\zeta_k](\cdot,\tau)
	\|_{L^\lambda(\R^n)} 
	\leq C' I_k^{(1)} + C' I_k^{(2)}, \\
	& 
	I_k^{(1)}:= \| (w \rho^\frac{1}{\lambda} 
	\psi_k^\frac{4}{\lambda} \zeta_k)_\tau
	\|_{L^2(\R^n\times [\tilde \tau-2,\tilde \tau])}, \\
	&I_k^{(2)}:= \| w \rho^\frac{1}{\lambda} 
	\psi_k^\frac{4}{\lambda} \zeta_k
	\|_{L^{(p+1)q}([\tilde \tau-2,\tilde \tau]; L^{p+1}(\R^n))}. 
\end{aligned}
\]
Here, by the H\"older inequality and $\supp \psi_k\subset B_1$, 
we used 
\[
	\| w \rho^\frac{1}{\lambda} 
	\psi_k^\frac{4}{\lambda} \zeta_k
	\|_{L^2(\R^n\times [\tilde \tau-2,\tilde \tau])}
	\leq 2^{\frac{1}{2}-\frac{1}{(p+1)q}} |B_1|^{\frac{1}{2}-\frac{1}{p+1}} I_k^{(2)}. 
\]

Let us estimate $I_k^{(1)}$ and $I_k^{(2)}$. 
Recall the $0$-extension of $w$ 
and $\zeta_k(\sigma)=0$ for $\sigma\leq \tilde \tau-1-2^{-(k-1)}$. 
By abbreviating $\iint \ldots 
= \int_{\tilde \tau-1-2^{-(k-1)}}^{\tilde \tau} \int_{\Omega(\sigma)} 
\ldots d\eta d\sigma$, 
\[
\begin{aligned}
	(I_k^{(1)})^2 \leq 
	2\iint |w_\tau|^2 \rho^\frac{2}{\lambda} 
	\psi_k^\frac{8}{\lambda} \zeta_k^2 
	+ 2\iint |w|^2 \rho^\frac{2}{\lambda} 
	\psi_k^\frac{8}{\lambda} |\partial_\tau \zeta_k|^2. 
\end{aligned}
\]
Since $\rho(\eta)=e^{-|\eta|^2/4}$ and 
$\psi_k=0$ in $\R^n\setminus \overline{B_{R/2^{k-1}}}$, we have 
\[
	\rho^\frac{2}{\lambda} 
	\psi_k^\frac{8}{\lambda}
	= ( e^{\frac{\lambda-2}{4\lambda}|\eta|^2} 
	\psi_k^\frac{8}{\lambda} ) \rho
	\leq C_k \psi_k^\frac{8}{\lambda} \rho 
	\leq C_k \rho. 
\]
This together with $\zeta_k\leq 1$, $|\partial_\tau\zeta_k|\leq C_k$, 
\eqref{eq:wtauL2} and \eqref{eq:wL2Linfty} yields  
$I_k^{(1)} \leq C_k$. 
As for $I_k^{(2)}$, 
by $\lambda<p+1$ and $\psi_{k-1}=1$ on $\supp \psi_k$, 
we see from \eqref{eq:assumpq2} that 
\[
\begin{aligned}
	(I_k^{(2)})^{(p+1)q} 
	&= \int_{\tilde \tau-2}^{\tilde \tau} 
	\left( \int_{\Omega(\sigma)} |w|^{p+1} \rho^\frac{p+1}{\lambda} 
	\psi_k^\frac{4(p+1)}{\lambda} \zeta_k^{p+1} d\eta \right)^q d\sigma \\
	&\leq 
	\int_{\tilde \tau-1-2^{-(k-1)}}^{\tilde \tau} 
	\left( \int_{\Omega(\sigma)} |w|^{p+1} \rho
	\psi_{k-1}^4 
	\chi_{\supp \psi_k} 
	d\eta \right)^q d\sigma \leq C_k, 
\end{aligned}
\]
where $\chi_{\supp \psi_k}$ is the characteristic function on $\supp \psi_k$. 
Note that $C_k>0$ is independent of $\tilde \tau$. 
Then, the desired estimate follows. 
\end{proof}

By \eqref{eq:p122tau}, \eqref{eq:inter} and Lemma \ref{lem:CLinter}, 
we have 
\begin{equation}\label{eq:midcalc123}
	\int_{\Omega(\tau)} |w|^{p+1} \rho \psi_k^4 d\eta
	\leq 
	C_k \| w_\tau(\cdot,\tau) 
	\|_{L_{\rho_k}^\frac{p+1}{p}}^\theta 
	\| w_\tau (\cdot,\tau) 
	\|_{L_{\rho_k}^2}^{1-\theta} 
	+ C_k 
\end{equation}
for $\tilde \tau-1-2^{-k}< \tau< \tilde \tau$ with  
$\tilde \tau\geq \tau_*+2$, where $L_{\rho_k}^{\lambda/(\lambda-1)}$ 
is the $L^{\lambda/(\lambda-1)}(\Omega(\tau))$ space 
with respect to the measure $d\rho_k= \rho\psi_k^4 d\eta$. 
We give further estimates.

\begin{lemma}\label{lem:cazeLio}
Suppose that \eqref{eq:assumpq2} holds.  
Then, for $\tilde \tau\geq \tau_*+2$, 
\[
\begin{aligned}
	&\int_{\tilde \tau-1-2^{-k}}^{\tilde \tau} \left( 
	\int_{\Omega(\sigma)} |w|^{p+1} \rho \psi_k^4 d\eta
	\right)^{q+\frac{1}{p+1}} d\sigma  \\
	&\leq 
	C_k \left( \int_{\tilde \tau-1-2^{-k}}^{\tilde \tau} 
	\left( \int_{\Omega(\sigma)} 
	| w_\tau|^\frac{p+1}{p}  \rho  \psi_k^4 
	d\eta \right)^{  \frac{p \beta}{p+1} } 
	d\sigma 
	\right)^\frac{1}{r} 
	+ C_k, 
\end{aligned}
\]
where $C_k>0$ is independent of $\tilde \tau$. 
The exponents $r, \beta>1$ are given 
in \eqref{eq:rrdashdef} and \eqref{eq:betadef1} below, respectively, 
and depend on $\eps$ in \eqref{eq:lamchoice}. 
Here, $\eps$ is chosen as a small constant 
depending only on $p$ and $q$.
\end{lemma}

\begin{proof}
Let $\tilde \tau\geq \tau_*+2$. 
By  \eqref{eq:midcalc123},  
\[
\begin{aligned}
	&\int \left( 
	\int_{\Omega(\sigma)} |w|^{p+1} \rho \psi_k^4 d\eta
	\right)^{q+\frac{1}{p+1}} d\sigma  \\
	&\leq 
	C_k \int  \| w_\tau(\cdot,\sigma) 
	\|_{L_{\rho_k}^\frac{p+1}{p}}^{\theta (q+\frac{1}{p+1})}
	\| w_\tau(\cdot,\sigma) 
	 \|_{L_{\rho_k}^2}^{(1-\theta) (q+\frac{1}{p+1})} d\sigma 
	+ C_k, 
\end{aligned}
\]
where $\int \ldots d\sigma
= \int_{\tilde \tau -1-2^{-k}}^{\tilde \tau} \ldots d\sigma$. 
Set 
\begin{equation}\label{eq:rrdashdef}
	r:= \frac{2(p+1)}{2(p+1)-(1-\theta)(pq+q+1)}, \quad 
	r':= \frac{2(p+1)}{(1-\theta)(pq+q+1)}. 
\end{equation}
By \eqref{eq:lamchoice} and \eqref{eq:thetadef}, 
we have $r,r'>1$. 
Then, \eqref{eq:wtauL2} and 
$d\rho_k= \rho\psi_k^4 d\eta$ yield  
\[
\begin{aligned}
	&\int \left( 
	\int_{\Omega(\sigma)} |w|^{p+1} \rho \psi_k^4 d\eta
	\right)^{q+\frac{1}{p+1}} d\sigma \\
	&\leq 
	C_k \left( \int  \| w_\tau(\cdot,\sigma) 
	\|_{L_{\rho_k}^\frac{p+1}{p}}^{\theta (q+\frac{1}{p+1}) r} d\sigma 
	\right)^\frac{1}{r} 
	\left( \int \| w_\tau(\cdot,\sigma) 
	 \|_{L_{\rho_k}^2}^2 d\sigma \right)^\frac{1}{r'} 
	+ C_k \\
	&\leq 
	C_k \left( \int  
	\left(  \int_{\Omega(\sigma)} 
	|w_\tau|^\frac{p+1}{p} \rho \psi_k^4 d\eta
	\right)^{\frac{pr\theta}{p+1}(q+\frac{1}{p+1})} d\sigma \right)^\frac{1}{r}
	+ C_k. 
\end{aligned}
\]

As for exponents, we define $\beta$ and compute 
from \eqref{eq:thetadef} and \eqref{eq:lamchoice} that 
\begin{equation}\label{eq:betadef1}
	\beta:=r\theta \left( q+\frac{1}{p+1} \right) 
	= \frac{(p+1)q(pq+q+1)}{p(pq+q+1)+p+1+c_\eps}, 
\end{equation}
where $c_\eps\in \R$ is a constant depending only on $p$, $q$ and $\eps$ 
and satisfying $c_\eps\to0$ as $\eps\to0$.
The condition $\beta>1$ is equivalent to $q/\beta<q$ and to 
\[
	\frac{p}{p+1}
	+ \frac{1}{pq+q+1}
	+ \frac{c_\eps}{(p+1)(pq+q+1)} <q. 
\]
This can be checked for all $p>1$ and $q\geq2$ if 
we take a small $\eps$ depending only on $p$ and $q$, 
and the desired estimate follows. 
\end{proof}

We estimate the right-hand side in the inequality 
in Lemma \ref{lem:cazeLio} by the maximal regularity 
from Giga--Matsui--Sasayama \cite[Theorem 2.1]{GMS04s}.

\begin{lemma}\label{lem:maximal}
Suppose that \eqref{eq:assumpq2} holds.  
Then for $\tilde \tau\geq \tau_*+2$, 
\[
\begin{aligned}
	&\int_{\tilde \tau-1-2^{-k}}^{\tilde \tau} 
	\left( \int_{\Omega(\sigma)} 
	| w_\tau|^\frac{p+1}{p}   \rho  \psi_k^4
	d\eta \right)^\frac{p \beta}{p+1}  d\sigma  \\
	&\leq 
	C_k \int_{\tilde \tau-1-2^{-(k-1)}}^{\tilde \tau} 
	\left( \int_{\Omega(\sigma)} 
	|\nabla w|^2 \rho \psi_{k-1}^4 d\eta \right)^\frac{\beta}{2} d\sigma + C_k, 
\end{aligned}
\]
where $C_k>0$ is independent of $\tilde \tau$. 
The exponent $\beta>1$ is given 
in \eqref{eq:betadef1} 
and depends on $\eps$ in \eqref{eq:lamchoice}. 
Here, $\eps$ is chosen as a small constant 
depending only on $p$ and $q$.
\end{lemma}

\begin{proof}
Let $\tilde \tau\geq \tau_*+2$. 
We use the cut-off function 
$\zeta_k$ in Lemma \ref{lem:CLinter}. 
Set 
\[
	v(\eta,\tau):= w(\eta,\tau)\rho^\frac{p}{p+1}(\eta) 
	\psi_k^\frac{4p}{p+1}(\eta) \zeta_k(\tau), 
	\quad \eta\in \overline{\Omega(\tau)}, \; \tau\geq \tau_*. 
\]
We observe that 
\begin{equation}\label{eq:wtasupzek11}
\begin{aligned}
	&|w_\tau|^\frac{p+1}{p} 
	\rho\psi_k^4 \zeta_k^\frac{p+1}{p}
	= 
	|w_\tau \rho^\frac{p}{p+1} 
	\psi_k^\frac{4p}{p+1} \zeta_k|^\frac{p+1}{p} 
	=|v_\tau - w \rho^\frac{p}{p+1} 
	\psi_k^\frac{4p}{p+1} (\zeta_k)_\tau|^\frac{p+1}{p} \\
	&\leq  
	C|v_\tau|^\frac{p+1}{p} 
	+ C_k|w|^\frac{p+1}{p} \rho 
	\chi_{\supp \psi_k}(\eta) \chi_{\supp \zeta_k}(\tau). 
\end{aligned}
\end{equation}
Moreover, $v$ satisfies 
\begin{equation}\label{eq:vprob}
\left\{ 
\begin{aligned}
	&v_\tau - \Delta v = F, 
	&&\eta \in\Omega(\tau), \; \tau > \tilde \tau -2, \\
	&v(\eta,\tau)= 0, &&\eta\in \partial \Omega(\tau), \; \tau > \tilde \tau -2, \\
	&v(\eta,\tilde \tau-2)= 0, 
	&&\eta\in \Omega(\tilde \tau-2), 
\end{aligned}
\right.
\end{equation}
where $F=F(\eta,\tau)$ is given by  
\[
	F:= 
	( w\rho^\frac{p}{p+1}\psi_k^\frac{4p}{p+1} \zeta_k )_\tau 
	-\Delta ( w\rho^\frac{p}{p+1}\psi_k^\frac{4p}{p+1} \zeta_k ). 
\]

Since $(p+1)/p>1$ and $\beta>1$, 
we can apply the maximal regularity estimate in 
Giga--Matsui--Sasayama \cite[Theorem 2.1]{GMS04s} 
to the solution $v$ of \eqref{eq:vprob} 
provided that $R$ satisfies the following smallness condition: 
\begin{equation}\label{eq:Rcond}
	0<R<\min\{1, R_0\}. 
\end{equation} 
Here, $R_0>0$ is the constant which is given in \cite[Theorem 2.1]{GMS04s} 
and depends only on $n$, $p$, $\Omega$ and $\beta$. 
Then the referred maximal regularity guarantees the existence of $C_0>0$ 
depending only on 
$n$, $p$, $\Omega$, $\beta$ and the length of the time interval 
$|[\tilde \tau-2, \tilde \tau]|=2$ such that 
\begin{equation}\label{eq:maximalvv}
\begin{aligned}
	\int_{\tilde \tau-2}^{\tilde \tau} 
	\left( \int_{\Omega(\sigma)} |v_\tau|^\frac{p+1}{p} d\eta 
	\right)^\frac{p \beta}{p+1}  d\sigma 
	\leq 
	C_0 \int_{\tilde \tau-2}^{\tilde \tau} 
	\left( \int_{\Omega(\sigma)} |F|^\frac{p+1}{p} d\eta 
	\right)^\frac{p \beta}{p+1} d\sigma. 
\end{aligned}
\end{equation}

Since $w$ satisfies \eqref{eq:weq}, we see that 
\begin{equation}\label{eq:Fcomputation}
\begin{aligned}
	F&= 
	\left( - \frac{1}{2} \eta\cdot \nabla w
	-\frac{1}{p-1} w + |w|^{p-1}w \right) 
	\rho^\frac{p}{p+1}\psi_k^\frac{4p}{p+1} \zeta_k \\
	&\quad 
	+w\rho^\frac{p}{p+1}\psi_k^\frac{4p}{p+1} \partial_\tau\zeta_k 
	-2\zeta_k\nabla  w\cdot \nabla (\rho^\frac{p}{p+1}\psi_k^\frac{4p}{p+1} ) \\
	&\quad 
	-w\zeta_k \Delta (\rho^\frac{p}{p+1}\psi_k^\frac{4p}{p+1} ), 
\end{aligned}
\end{equation}
and so by $\supp \psi_k\subset B_R$, \eqref{eq:nabpsibd} and 
$\rho(\eta)=e^{-|\eta|^2/4}$, 
\begin{equation}\label{eq:Fcalces}
\begin{aligned}
	&|F|^\frac{p+1}{p} \\
	&\leq 
	\left( |w|^p \rho^\frac{p}{p+1} \psi_k^\frac{4p}{p+1} \zeta_k
	+ C_k (|w|+|\nabla w|) \rho^\frac{p}{p+1} \chi_{\supp \psi_k} 
	\chi_{\supp \zeta_k}
	\right)^\frac{p+1}{p}  \\
	&\leq 
	C |w|^{p+1} \rho \psi_k^4 \chi_{\supp \zeta_k}
	+ C_k ( |w|^\frac{p+1}{p}  +|\nabla w|^\frac{p+1}{p} )
	\rho \chi_{\supp \psi_k}
	\chi_{\supp \zeta_k}. 
\end{aligned}
\end{equation}
From $\zeta_k=1$ on $[\tilde \tau-1-2^{-k}, \infty)$, 
\eqref{eq:wtasupzek11}, 
$\psi_{k-1}^4=1$ on $\supp \psi_k$, 
\eqref{eq:maximalvv}, \eqref{eq:Fcalces} and 
$\supp \zeta_k \subset [\tilde \tau-1-2^{-(k-1)}, \infty)$, 
it follows with $\int \ldots 
= \int_{\Omega(\sigma)} \ldots d\eta$ 
and the abbreviation of $d\sigma$ that 
\[
\begin{aligned}
	&\int_{\tilde \tau-1-2^{-k}}^{\tilde \tau} 
	\left( \int 
	|w_\tau|^\frac{p+1}{p} \rho \psi_k^4  \right)^\frac{p \beta}{p+1}  
	\leq 
	\int_{\tilde \tau-2}^{\tilde \tau} 
	\left( \int 
	|w_\tau|^\frac{p+1}{p} \rho \psi_k^4  
	\zeta_k^\frac{p+1}{p} \right)^\frac{p \beta}{p+1}  \\
	&\leq 
	C\int_{\tilde \tau-2}^{\tilde \tau} 
	\left( \int 
	|v_\tau|^\frac{p+1}{p}  \right)^\frac{p \beta}{p+1}  
	+ C_k\int_{\tilde \tau-2}^{\tilde \tau} 
	\left( \int 
	|w|^\frac{p+1}{p} \rho \psi_{k-1}^4 \chi_{\supp \zeta_k}
	 \right)^\frac{p \beta}{p+1} \\
	&\leq 
	C\int_{\tilde \tau-1-2^{-(k-1)}}^{\tilde \tau} 
	\left( \int 
	|w|^{p+1} \rho \psi_k^4  \right)^\frac{p \beta}{p+1}  \\
	&\quad 
	+ C_k\int_{\tilde \tau-1-2^{-(k-1)}}^{\tilde \tau} 
	\left( \int 
	( |w|^\frac{p+1}{p}+ |\nabla w|^\frac{p+1}{p}) \rho \psi_{k-1}^4
	\right)^\frac{p \beta}{p+1} 
	=: C J_k^{(1)} + C_k J_k^{(2)}. 
\end{aligned}
\]

By \eqref{eq:betadef1}, the condition 
$p\beta/(p+1)<q$ is satisfied for $p>1$ and $q\geq2$ 
if $\eps$ is small depending only on $p$ and $q$, since 
$p\beta/(p+1)<q$ is equivalent to 
\begin{equation}\label{eq:pp1qbeta}
	\frac{p}{(p+1)q}<\frac{1}{\beta} = 
	\frac{p}{(p+1)q} + \frac{1}{(pq+q+1)q} 
	+ \frac{c_\eps}{(p+1)q(pq+q+1)} . 
\end{equation}
Therefore, the H\"older inequality, $\psi_k\leq \psi_{k-1}$ 
and \eqref{eq:assumpq2} yield 
\[
\begin{aligned}
	J_k^{(1)} 
	\leq 
	C_k \left( \int_{\tilde \tau-1-2^{-(k-1)}}^{\tilde \tau} 
	\left( \int
	|w|^{p+1} \rho \psi_{k-1}^4
	d\eta \right)^q d\sigma \right)^\frac{p\beta}{(p+1)q} 
	\leq C_k. 
\end{aligned}
\]
On the other hand, by $(p+1)/p<2$, 
the H\"older inequality and \eqref{eq:wL2Linfty}, 
\[
\begin{aligned}
	J_k^{(2)}&\leq 
	\int_{\tilde \tau-1-2^{-(k-1)}}^{\tilde \tau} 
	\left( \int |\nabla w|^\frac{p+1}{p} \rho \psi_{k-1}^4
	d\eta \right)^\frac{p \beta}{p+1}  d\sigma 
	+ C_k  \\
	&\leq 
	C \int_{\tilde \tau-1-2^{-(k-1)}}^{\tilde \tau} 
	\left( \int |\nabla w|^2 \rho \psi_{k-1}^4
	d\eta \right)^\frac{\beta}{2}  d\sigma 
	+ C_k. 
\end{aligned}
\]
Hence we obtain the desired estimate. 
\end{proof}

We fix $\eps>0$ in \eqref{eq:lamchoice} 
so small that $2<\lambda<p+1$ and 
Lemmas \ref{lem:cazeLio} and \ref{lem:maximal} hold, 
where $\eps$ can be determined by $p$ and $q$ only. 
Then, we are now in a position to prove Proposition \ref{pro:boot}.

\begin{proof}[Proof of Proposition \ref{pro:boot}]
Let $\tilde \tau\geq \tau_*+2$. 
We see from $p>1$ and \eqref{eq:pp1qbeta} that 
\[
	\frac{\beta}{2} < \frac{p\beta}{p+1} < q. 
\]
Then by Lemmas \ref{lem:cazeLio} and \ref{lem:maximal} and 
the H\"older inequality, 
\[
\begin{aligned}
	&\int_{\tilde \tau-1-2^{-k}}^{\tilde \tau}  \left( 
	\int_{\Omega(\sigma)} |w|^{p+1} \rho \psi_k^4 d\eta
	\right)^{q+\frac{1}{p+1}} d\sigma  \\
	&\leq 
	C_k \left( 
	\int_{\tilde \tau-1-2^{-(k-1)}}^{\tilde \tau} 
	\left( \int_{\Omega(\sigma)} 
	|\nabla w|^2 \rho \psi_{k-1}^4 d\eta \right)^\frac{\beta}{2} d\sigma
	\right)^\frac{1}{r} + C_k \\
	&\leq 
	C_k \left( 
	\int_{\tilde \tau-1-2^{-(k-1)}}^{\tilde \tau} 
	\left( \int_{\Omega(\sigma)} 
	\frac{|\nabla w|^2}{2} \rho \psi_{k-1}^4 d\eta \right)^q d\sigma
	\right)^\frac{\beta}{2 q r} 
	+ C_k. 
\end{aligned}
\]
From \eqref{eq:cElocdef}, 
Lemma \ref{lem:locEunibdd} and \eqref{eq:assumpq2}, 
it follows that 
\[
\begin{aligned}
	&\int_{\tilde \tau-1-2^{-(k-1)}}^{\tilde \tau} 
	\left( \int_{\Omega(\sigma)} 
	\frac{|\nabla w|^2}{2} \rho \psi_{k-1}^4 d\eta \right)^q d\sigma \\
	&\leq 
	\int_{\tilde \tau-1-2^{-(k-1)}}^{\tilde \tau} 
	\left( 
	\cE_{(\tilde x, T)}^{\psi_{k-1}} (\sigma)
	+ \frac{1}{p+1} \int_{\Omega(\sigma)} |w|^{p+1} \rho \psi_{k-1}^4 d\eta 
	\right)^q d\sigma 
	\leq C_k. 
\end{aligned}
\]
Hence we obtain 
\[
	\int_{\tilde \tau-1-2^{-k}}^{\tilde \tau}  \left( 
	\int_{\Omega(\sigma)} |w|^{p+1} \rho \psi_k^4 d\eta
	\right)^{q+\frac{1}{p+1}} d\sigma \leq C_k, 
\]
where $C_k>0$ is independent of $\tilde \tau$. 
Taking the supremum over $\tilde \tau\geq \tau_*+2$ 
yields the desired estimate \eqref{eq:assumpqp1}. 
The proof is complete. 
\end{proof}

As a final preparation for proving Theorem \ref{th:main}, 
we give the following estimate 
based on Propositions \ref{pro:base} (the basis of the inductive steps) 
and \ref{pro:boot} (each of the inductive steps). 

\begin{proposition}\label{pro:lessp1}
Let $p>1$ and $1<\lambda<p+1$. 
Then there exist $0<R_*<1$ depending only on 
$n$, $p$, $\lambda$ and $\Omega$
and $C>0$ depending only on 
$n$, $p$,  $\lambda$, $\Omega$, $\tau_*$, $C_*$ and $C_*'$
such that for $\tilde x\in \Omega$, 
\[
	\sup_{\tau\geq \tau_*+1} 
	\int_{\Omega(\tau)\cap B_{R_*}(0)} 
	|w_{(\tilde x,T)}(\eta,\tau)|^\lambda \rho d\eta \leq  C, 
\]
equivalently, 
\[
	\sup_{ T-\delta^2/e \leq t < T } 
	(T-t)^\frac{\lambda}{p-1}
	\int_{\Omega \cap B_{R_*\sqrt{T-t}}(\tilde x)}  
	|u(x,t)|^\lambda 
	K_{(\tilde x,T)}(x,t) dx \leq  C, 
\]
where $\tau_*$, $C_*$ and $C_*'$ are given in \eqref{eq:stars}, 
$\delta$ is given in Proposition \ref{pro:qmono} and 
they are independent of $\tilde x\in \Omega$. 
\end{proposition}

\begin{proof}
Let $p>1$ and $1<\lambda<p+1$. 
We take $q_*\geq2$ and $\lambda_*>\lambda$ 
depending only on $p$ and $\lambda$ 
such that 
\[
	\max\{2,\lambda\}< \lambda_*< p+1 - \frac{p-1}{q_*+1}. 
\]
By Proposition \ref{pro:base} and by inductively 
applying Proposition \ref{pro:boot} 
finitely many times, there exists $k_*\geq 1$ 
depending only on $p$ and $q_*$ such that 
\[
	\sup_{\tau\geq \tau_*+2}\int_{\tau-1-2^{-{k_*}}}^{\tau} 
	\left( \int_{\Omega(\sigma)} |w_{(\tilde x,T)}|^{p+1} \rho \psi_{k_*}^4 d\eta 
	\right)^{q_*} d\sigma \leq C,  
\]
where $C>0$ depends only on 
$k_*$, $n$, $p$, $\lambda$, $\Omega$, $R$, $\tau_*$, $C_*$ and $C_*'$. 
Then by the same argument as in Lemma \ref{lem:CLinter}, 
\[
	\sup_{\tau\in[\tilde \tau-1-2^{-k_*},\tilde \tau]} 
	\| [w \rho^\frac{1}{\lambda_*} \psi_{k_*}^\frac{4}{\lambda_*}]
	(\cdot,\tau)\|_{L^{\lambda_*}(\Omega(\tau))} 
	\leq  C 
	\quad \mbox{ for }\tilde \tau\geq \tau_*+2, 
\]
where $C>0$ is independent of $\tilde \tau$. 
In particular, 
\[
	\| [w \rho^{1/\lambda_*} \psi_{k_*}^{4/\lambda_*}]
	(\cdot,\tau)\|_{L^{\lambda_*}(\Omega(\tau))} \leq  C 
	\quad \mbox{ for }\tau\geq \tau_*+1. 
\]
This together with $\psi_{k_*}=1$ in $B_{R/2^{k_*}}(0)$ 
and $\lambda_*>\lambda$ gives the desired estimate with some constant 
$C>0$ depending only on 
$k_*$, $n$, $p$, $\lambda$, $\Omega$, $R$, $\tau_*$, $C_*$ and $C_*'$. 
As for dependence, we observe that $k_*$ 
(resp. $R$ in \eqref{eq:Rcond}) 
is determined by $p$ and $\lambda$ (resp. $n$, $p$, $\lambda$ and $\Omega$) only. 
The proof is complete. 
\end{proof}

Finally, we show the type I rate in the energy subcritical range.

\begin{proof}[Proof of Theorem \ref{th:main}]
Let $p<p_S$ and $\tilde x\in\Omega$. 
We prove the theorem by deducing an $L^\infty$-estimate of $w$. 
Since $p<p_S$ is equivalent to 
$n(p-1)/2<p+1$, we can choose $q, r>1$ such that 
\[
	\max\left\{ \frac{n}{2}, 1,\frac{1}{p-1} 
	\right\} < q < \frac{p+1}{p-1}, 
	\quad 
	\frac{1}{r} + \frac{n}{2q} < 1. 
\]
Then by \eqref{eq:wL2Linfty} and Proposition \ref{pro:lessp1} 
with $\lambda=(p-1)q$, 
\[
\begin{aligned}
	&\int_{\tau'-1}^{\tau'}
	\int_{\Omega(\sigma)\cap B_{R_*}} 
	|w_{(\tilde x,T)}(\eta,\sigma)|^2 \rho d\eta d\sigma \leq  C, \\
	&\int_{\tau'-1}^{\tau'}
	\left( \int_{\Omega(\sigma)\cap B_{R_*}} 
	\left( |w_{(\tilde x,T)}(\eta,\sigma)|^{p-1} \right)^q \rho d\eta 
	\right)^\frac{r}{q} d\sigma  \leq  C, 
\end{aligned}
\]
for $\tau' \geq \tau_*+2$, 
where $C>0$ is independent of $\tilde x\in \Omega$. 
By the same argument as in Giga--Kohn \cite[Subsection 3A]{GK87} 
based on the interior and boundary regularity 
(more precisely, see line 9 on page 14 through line 16 on page 16 
and Remark 3.6 in \cite{GK87}), 
there exists $C>0$  independent of $\tilde x\in \Omega$ such that 
\[
	|w_{(\tilde x,T)}(0,\tau)| \leq C
\] 
for $\tau'-1/2<\tau<\tau'$ and $\tau'\geq \tau_*+2$. 
Thus, from \eqref{eq:backw} and from 
$u\in L^\infty((0,T-e^{-(\tau_*+2)}); L^\infty(\Omega))$ by \eqref{eq:regu}, 
it follows that 
\[
	(T-t)^\frac{1}{p-1} |u(\tilde x,t) |
	\leq C
\]
for $\tilde x\in \Omega$  and $0<t<T$, 
where $C>0$ is independent of $\tilde x$. 
Hence $u$ satisfies the desired type I blow-up estimate. 
The proof is complete. 
\end{proof}

\section*{Acknowledgments}
The authors wish to express their gratitude 
to Professors Philippe Souplet and Yifu Zhou 
for their valuable comments.
The first author was supported in part by JSPS KAKENHI 
Grant Number 23K20803. 
The second author was supported in part 
by JSPS KAKENHI Grant Numbers 22KK0035, 23K12998 and 23K22402.
The third author was supported in part 
by JSPS KAKENHI Grant Numbers 22KK0035, 23K13005 and 25KJ0013.

\section*{Use of AI}
ChatGPT was used to check grammar and improve the English wording of the manuscript.

\end{document}